\documentclass[11pt,a4paper]{amsart}

\usepackage{mathtools}
\usepackage{graphicx}
\usepackage{amsthm}
\usepackage{amssymb}
\usepackage{caption}
\usepackage{hyperref}
\usepackage{nicematrix}
\usepackage{enumitem}
\usepackage{booktabs}
\usepackage[dvipsnames]{xcolor}
\usepackage[left=30mm, right=30mm]{geometry}

\usepackage{tikz}
\usepackage{tikz-cd}
\usetikzlibrary{calc}
\usetikzlibrary{decorations.pathmorphing}
\usetikzlibrary{decorations.pathreplacing}

\theoremstyle{definition}
\newtheorem{definition}{Definition}[section]

\theoremstyle{plain}

\newtheorem{lemma}[definition]{Lemma}

\newtheorem{theorem}[definition]{Theorem}

\numberwithin{equation}{section}

\tikzset{
	vertex/.style={
		circle,
		minimum size=1.8mm,
		fill,
		inner sep=0,
		outer sep=0,
	},
	edge/.style={
		line width=.25mm,
	}
}

\ExplSyntaxOn

\cs_new_protected:Npn \cycle #1
{
	\seq_set_split:Nnn\l_tmpa_seq{~}{#1}
	(
	\seq_use:Nn\l_tmpa_seq{\,} 
	)
}

\cs_new_protected:Npn \chain #1
{
	\seq_set_split:Nnn\l_tmpa_seq{~}{#1}
	(
	\seq_use:Nn\l_tmpa_seq{\,} 
	\rangle
}

\ExplSyntaxOff

\newcommand*\defterm{\emph}

\DeclareMathOperator{\dom}{Dom}
\DeclareMathOperator{\im}{Im}

\newcommand*\diam[1]{\mathrm{diam}(#1)}
\newcommand*\dist[3]{d_{#1}(#2,#3)}

\newcommand*\commgraph[1]{\mathcal{G}(#1)}

\newcommand*\centre[1]{Z(#1)}
\newcommand*\idemp[1]{E(#1)}

\DeclarePairedDelimiter{\abs}{\lvert}{\rvert}

\DeclarePairedDelimiter{\parens}{\lparen}{\rparen}
\DeclarePairedDelimiter{\bracks}{\lbrack}{\rbrack}
\DeclarePairedDelimiter{\braces}{\{}{\}}

\DeclarePairedDelimiter{\set}{\{}{\}}
\DeclarePairedDelimiterX{\gset}[2]{\{}{\}}{\,#1:#2\,}

\newcommand*\Xn{\set{1,\ldots,n}}
\newcommand*\X[1]{\set{1,\ldots,#1}}

\newcommand*{\ptr}[1]{\mathcal{P}(#1)}
\newcommand*{\tr}[1]{\mathcal{T}(#1)}

\newcommand*{\psym}[1]{\mathcal{I}(#1)}
\newcommand*{\alt}[1]{\mathcal{A}(#1)}

\newcommand*{\sym}[1]{\mathcal{S}(#1)}

\newcommand*\id[1]{\mathrm{id}_{#1}}

\newcommand*{\graphptr}[2]{\mathcal{G}^X(#1,#2)}

\makeatletter
\newcommand*{\sizeddelimiter}[2]{\bBigg@{#1}#2}
\makeatother

\allowdisplaybreaks

\begin{document}

\title{Diameters of commuting graphs of partial transformation semigroups}

\author{Tânia Paulista}
\address[T. Paulista]{%
Center for Mathematics and Applications (NOVA Math) \& Department of Mathematics\\
NOVA School of Science and Technology\\
NOVA University of Lisbon\\
2829--516 Caparica\\
Portugal
}
\email{%
tpl.paulista@gmail.com
}
\thanks{This work is funded by national funds through the FCT -- Fundação para a Ciência e a Tecnologia, I.P., under the scope of the projects UID/297/2025 and UID/PRR/297/2025 (Center for Mathematics and Applications - NOVA Math). The author is also funded by national funds through the FCT -- Fundação para a Ciência e a Tecnologia, I.P., under the scope of the studentship 2021.07002.BD}

\thanks{The author is thankful to her supervisors António Malheiro and Alan J. Cain for all the support, encouragement and guidance; and also for reading a draft of this paper}


\begin{abstract}
	Let $X$ be a finite set. We determine the diameter of the commuting graph of the partial transformation semigroup $\ptr{X}$ on $X$ and show that it coincides with the diameter of the commuting graph of the transformation semigroup $\tr{X}$ on $X$, which was previously determined by Araújo, Kinyon and Konieczny. This proves the existence of a semigroup $S$ and of a proper subsemigroups $T$ of $S$ such that the diameters of the commuting graphs of $S$ and $T$ are equal.
\end{abstract}

\maketitle

\section{Introduction}

The commuting graph of a semigroup is a simple graph whose vertices are elements of the semigroup, and where vertices are adjacent precisely when the corresponding elements of the semigroup commute.

Commuting graphs were first introduced in 1955 by Brauer and Fowler \cite{First_paper_commuting_graphs}, exclusively for groups. It was only in 2011 that Araújo, Kinyon and Konieczny \cite{Commuting_graph_T_X} extended the notion of commuting graphs to semigroups. 

Over the years, commuting graphs have proved to be powerful tools for solving problems in group and semigroup theory. For example, they were involved in the discovery of three sporadic simple groups (now known as the Fischer groups) \cite{Sporadic_simple_groups}. In 1983, Bertram \cite{Importance_commuting_graphs_1} used them to determine an upper bound for the size of the abelian subgroups of a finite group. There are also several papers \cite{Importance_commuting_graphs_2, Importance_commuting_graphs_3, Importance_commuting_graphs_5, Importance_commuting_graphs_6, Importance_commuting_graphs_4} where commuting graphs were used to study the quotients of the multiplicative group of a finite dimensional division algebras. In 2011, Araújo, Kinyon and Konieczny \cite{Commuting_graph_T_X} studied minimal left paths in commuting graphs to settle a conjecture posed by Schein \cite{Schein_conjecture} concerning the characterization of
$r$-semisimple bands.

There are multiple ways to investigate commuting graphs: we highlight the study of the groups whose commuting graphs are perfect \cite{Commuting_graphs_perfect}, split \cite{Commuting_graphs_groups_split}, chordal and cographs \cite{Graphs_arise_as_commuting_graphs_groups}; the identification of the simple graphs that are commuting graphs of groups/semigroups \cite{Graphs_arise_as_commuting_graphs_groups, Graphs_that_arise_as_commuting_graphs_of_semigroups, Graphs_that_arise_as_commuting_graphs_of_semigroups_2}; and the determination of the sets of possible values for the diameter/clique number/chromatic number/girth/knit degree of the commuting graph of a semigroup belonging to a particular class of semigroups \cite{Commuting_graph_T_X, Graphs_that_arise_as_commuting_graphs_of_semigroups, Group_whose_commuting_graph_has_diameter_n, Graphs_that_arise_as_commuting_graphs_of_semigroups_2, Completely_0-simple_paper, Completely_simple_semigroups_paper, Commuting_graphs_inverse_completely_regular}. 

Another popular line of research of commuting graphs, and the most relevant one for this paper, is the determination of properties of the commuting graphs of important groups and semigroups. The study of the commuting graph of the symmetric group $\sym{X}$ on a finite set $X$ started with the paper \cite{Symmetric_group}, which led to the determination of the clique number of the commuting graph. Other papers focused on determining the diameter of the commuting graph of $\sym{X}$ \cite{Commuting_graph_I_X, Diameter_commuting_graph_symmetric_group, Commuting_graph_symmetric_alternating_groups}, as well as its girth and minimum degree \cite{Commuting_graph_symmetric_alternating_groups}. Various properties of the commuting graph of the alternating group $\alt{X}$ on a finite set $X$ were also determined: the clique number \cite{Alternating_group}, diameter, girth and minimum degree \cite{Commuting_graph_symmetric_alternating_groups}. In 2011, Araújo, Kinyon and Konieczny \cite{Commuting_graph_T_X} obtained the diameter of the commuting graph of the transformation semigroup $\tr{X}$ on a finite set $X$, as well as the diameters of the commuting graphs of the ideals of $\tr{X}$. More recently, the present author \cite{Largest_commutative_T_X_P_X} studied the clique number, girth and knit degree of this commuting graph. In 2015, Araújo, Bentz and Konieczny \cite{Commuting_graph_I_X} obtained the clique number and diameter of the commuting graph of the symmetric inverse semigroup $\psym{X}$ on a finite set $X$, and obtained the diameters of the commuting graphs of the ideals of $\psym{X}$. In the paper \cite{Largest_commutative_T_X_P_X}, the present author investigated the clique number, girth and knit degree of the commuting graph of the partial transformation semigroup $\ptr{X}$ on a finite set $X$. All the properties we mentioned in this paragraph are mostly dependent on $\abs{X}$: for example the various diameters depend on $\abs{X}$ in ways that involve primality.

This paper offers a contribution to the continuing research on the commuting graph of $\ptr{X}$. After a brief detour through some basic concepts concerning simple graphs and commuting graphs (Section~\ref{Preliminaries}), we obtain the diameter of this graph and verify that it matches the diameter of the commuting graph of $\tr{X}$, which was previously determined by Araújo, Kinyon and Konieczny (Section~\ref{sec: diameter P(X)}). In order to obtain the diameter, we introduce a new combinatorial technique that allow us to verify whether two full transformations share an adjacent vertex lying in $\ptr{X}\setminus\tr{X}$.

This paper is based on Chapter 4 of the author's Ph.D. thesis \cite{My_thesis}.

\section{Preliminaries} \label{Preliminaries}

\subsection{Simple graphs}\label{sec: graphs}

A \defterm{simple graph} $G=(V,E)$ consists of a non-empty set $V$ --- whose elements are called \defterm{vertices} --- and a set $E$ --- whose elements are called \defterm{edges} --- formed by $2$-subsets of $V$. If $x$ and $y$ are vertices of $G$ such that $\set{x,y}\in E$, then we say that $x$ and $y$ are \defterm{adjacent}. Throughout this subsection we will assume that $G=(V,E)$ is a simple graph.

A simple graph $H=\parens{V',E'}$ is a \defterm{subgraph} of $G$ if $V'\subseteq V$ and $E'\subseteq E$. Note that, since $H$ is a simple graph, the elements of $E'$ are $2$-subsets of $V'$.

Given $V'\subseteq V$, the \defterm{subgraph induced by $V'$} is the subgraph of $G$ whose set of vertices is $V'$ and where two vertices are adjacent if and only if they are adjacent in $G$ (that is, the set of edges of the induced subgraph is $\braces{\braces{x,y}\in E: x,y\in V'}$).


A \defterm{path} in $G$ from a vertex $x$ to a vertex $y$ is a sequence of pairwise distinct vertices (except, possibly, $x$ and $y$) $x=x_1,x_2,\ldots,x_n=y$ such that $\braces{x_1,x_2}, \braces{x_2,x_3},\ldots, \braces{x_{n-1},x_n}$ are pairwise distinct edges of $G$. The \defterm{length} of the path is the number of edges of the path; thus, the length of our example path is $n-1$. Whenever we want to mention a path, we will write that $x=x_1-x_2-\cdots-x_n=y$ is a path (instead of writing that $x=x_1,x_2,\ldots,x_n=y$ is a path). The \defterm{distance} between the vertices $x$ and $y$, denoted $\dist{G}{x}{y}$, is the length of a shortest path from $x$ to $y$. If there is no such path between the vertices $x$ and $y$, then the distance between $x$ and $y$ is defined to be infinity, that is, $\dist{G}{x}{y}=\infty$.


We say that $G$ is \defterm{connected} if for all vertices $x,y\in V$ there is a path from $x$ to $y$. We can partition $V$, the vertex set of $G$, into several non-empty sets $V_1,\ldots,V_n$ such that
\begin{enumerate}
	\item For all $i\in\Xn$ and vertices $x,y\in V_i$ there is a path from $x$ to $y$.
	
	\item For all distinct $i,j\in\Xn$ and $x\in V_i$ and $y\in V_j$ there is no path from $x$ to $y$.
\end{enumerate}
Then each subgraph of $G$ induced by $V_i$, where $i\in\Xn$, is connected and we call it a \defterm{connected component} of $G$. It is clear that $G$ is connected if and only if $G$ contains exactly one connected component.


The \defterm{diameter} of $G$, denoted $\diam{G}$, is the maximum distance between vertices of $G$, that is, $\diam{G}=\max\gset{\dist{G}{x}{y}}{x,y\in V}$. We notice that the diameter of $G$ is finite if and only if $G$ is connected.

If $x$ and $y$ are vertices of $G$, then we are going to use the notation $x\sim y$ to mean that either $x=y$ or $\set{x,y}\in E$. Note that if $x_1-x_2-\cdots-x_n$ is a path, then we have $x_1\sim x_2 \sim\cdots\sim x_n$. However, if we have $x_1\sim x_2 \sim\cdots\sim x_n$, then that sequence of vertices does not necessarily form a path because there might exist distinct $i,j\in\Xn$ such that $x_i=x_j$.



\subsection{Commuting graphs and extended commuting graphs}
\label{sec: (extended) commgraph}

Recall that the \defterm{center} of a semigroup $S$ is the set
\begin{displaymath}
	Z(S)= \left\{x\in S: xy=yx \text{ for all } y\in S\right\}.
\end{displaymath}

The \defterm{commuting graph} of a finite non-commutative semigroup $S$, denoted $\commgraph{S}$, is the simple graph whose set of vertices is $S\setminus Z(S)$ and where two distinct vertices $x,y\in S\setminus Z(S)$ are adjacent if and only if $xy=yx$.



It follows from the definition of commuting graph that, for all vertices $x$ and $y$ of $\commgraph{S}$, we have $x\sim y$ if and only if $xy=yx$. Moreover, we observe that in the definition of commuting graph the semigroup must be non-commutative because, otherwise, we would obtain an empty vertex set. Furthermore, the definition also implies that $\diam{\commgraph{S}}\geqslant 2$ because, since $S$ must be non-commutative, then there exist $x,y\in S$ such that $xy\neq yx$, which implies that $\diam{\commgraph{S}}\geqslant\dist{\commgraph{S}}{x}{y}>1$.

%
%
%
%



\section{Diameters of commuting graphs of partial transformation semigroups}\label{sec: diameter P(X)}

Throughout this section we are going to assume that $X$ is a finite set. The \defterm{partial transformation semigroup} on $X$, denoted $\ptr{X}$, is the semigroup formed by all the partial transformations on $X$ (that is, all the functions whose domain and image are both contained in $X$) and whose multiplication is the composition of functions. The \defterm{transformation semigroup} on $X$, denoted $\tr{X}$, is the semigroup formed by all the (full) transformations on $X$ (that is, all the functions whose domain is $X$ and whose image is contained in $X$) and whose multiplication is the composition of functions.

The aim of this section is to determine the diameter of the commuting graph of $\ptr{X}$: we will see that, when $\abs{X}\geqslant 2$, the diameter of $\commgraph{\ptr{X}}$ is equal to the diameter of $\commgraph{\tr{X}}$, which was previously determined by Araújo, Kinyon and Konieczny \cite{Commuting_graph_T_X}. More specifically, we will see that $\commgraph{\ptr{X}}$ is connected (that is, that $\diam{\commgraph{\ptr{X}}}$ is finite) if and only if $\abs{X}$ is not prime. In this case we have that $\diam{\commgraph{\ptr{X}}}=4$, if $\abs{X}=4$, and $\diam{\commgraph{\ptr{X}}}=5$, if $\abs{X}\neq 4$.

We can easily verify that $\centre{\ptr{X}}=\set{\emptyset,\id{X}}$. This implies that $\ptr{X}$ is non-commutative if and only if $\ptr{X}\neq\set{\emptyset,\id{X}}$, which happens precisely when $\abs{X}\geqslant 2$. Thus $\commgraph{\ptr{X}}$ is only defined when $\abs{X}\geqslant 2$. Moreover, the fact that $\centre{\ptr{X}}=\set{\emptyset,\id{X}}$ also implies that there are only two partial transformations that are not vertices of $\commgraph{\ptr{X}}$: $\emptyset$ and $\id{X}$.

In Theorem~\ref{P(X): diameter} we will see that $\diam{\commgraph{\ptr{X}}}=\infty$ (that is, that $\commgraph{\ptr{X}}$ is not connected) if and only if $\abs{X}$ is prime. In the meantime, we take steps to obtain the diameter of $\commgraph{\ptr{X}}$ when $\abs{X}$ is composite. The reasoning is divided into two parts. First, we obtain an upper bound for the distance between vertices of $\commgraph{\ptr{X}}$; and then we find two partial transformations of $\ptr{X}$ whose distance in $\commgraph{\ptr{X}}$ matches the upper bound we obtained.

We begin with the determination of an upper bound for the distance between vertices of $\commgraph{\ptr{X}}$. In Lemma~\ref{P(X): upper bound diam, a,b€P(X)-(T(X) U 0)} we consider the distance between two partial transformations that are not full transformations, in Lemma~\ref{P(X): upper bound diam, a€T(X)-S(X) b€P(X)-(T(X) U 0)} we consider the distance between a full transformation that is not a permutation and a partial transformation that is not a full transformation, and in Lemma~\ref{P(X): upper bound diam, a€S(X)-1 b€P(X)-(T(X) U 0)} we consider the distance between a permutation and a partial transformation that is not a full transformation. We do not need to find an upper bound for the distance between full transformations because in \cite[Theorem 2.22]{Commuting_graph_T_X} the authors determined the diameter of $\commgraph{\tr{X}}$, which is the subgraph of $\commgraph{\ptr{X}}$ whose vertices are precisely the vertices of $\commgraph{\ptr{X}}$ that are full transformations. This means that the diameter of $\commgraph{\tr{X}}$ is an upper bound for the distance between full transformations (in $\commgraph{\ptr{X}}$).

Below we provide the diameter of $\commgraph{\tr{X}}$. We will see that the corresponding result for $\commgraph{\ptr{X}}$ is exactly parallel.

\begin{theorem}[{\cite[Theorem 2.22]{Commuting_graph_T_X}}]\label{P(X): diam T(X)}
	Suppose that $\abs{X}\geqslant 2$. Then
	\begin{enumerate}
		\item If $\abs{X}$ is prime, then $\commgraph{\tr{X}}$ is not connected.
		
		\item If $\abs{X}=4$, then $\diam{\commgraph{\tr{X}}}=4$.
		
		\item If $\abs{X}\geqslant 6$ is composite, then $\diam{\commgraph{\tr{X}}}=5$.
	\end{enumerate}
\end{theorem}

\begin{lemma}\label{P(X): upper bound diam, a,b€P(X)-(T(X) U 0)}
	Suppose that $\abs{X}\geqslant 4$. Let $\alpha,\beta\in\ptr{X}\setminus\parens{\tr{X}\cup\set{\emptyset}}$. Then $\dist{\ptr{X}}{\alpha}{\beta}\leqslant 4$.
\end{lemma}

\begin{proof}
	Due to the fact that $\alpha,\beta\notin\tr{X}$, we have that the sets $X\setminus\dom\alpha$, $X\setminus\im\alpha$, $X\setminus\dom\beta$ and $X\setminus\im\beta$ are non-empty. Let $x'\in X\setminus\dom\alpha$, $x\in X\setminus\im\alpha$, $y'\in X\setminus\dom\beta$ and $y\in X\setminus\im\beta$. Then we have
	\begin{displaymath}
		\alpha\sim\begin{pmatrix}
			x\\x'
		\end{pmatrix} \quad \text{and} \quad \beta\sim\begin{pmatrix}
			y\\y'
		\end{pmatrix}
	\end{displaymath}
	since all the relevant products are equal to the empty map. Moreover, it follows from the fact that $\abs{X}\geqslant 4$ that the sets $X\setminus\set{x,y}$ and $X\setminus\set{x',y'}$ are non-empty. Let $z'\in X\setminus\set{x,y}$ and $z\in X\setminus\set{x',y'}$. Hence
	\begin{gather*}
		\begin{pmatrix}
			z\\z'
		\end{pmatrix}\begin{pmatrix}
			x\\x'
		\end{pmatrix}=\emptyset=\begin{pmatrix}
			x\\x'
		\end{pmatrix}\begin{pmatrix}
			z\\z'
		\end{pmatrix}\\
		\shortintertext{and}
		\begin{pmatrix}
			z\\z'
		\end{pmatrix}\begin{pmatrix}
			y\\y'
		\end{pmatrix}=\emptyset=\begin{pmatrix}
			y\\y'
		\end{pmatrix}\begin{pmatrix}
			z\\z'
		\end{pmatrix}.
	\end{gather*}
	Therefore we have
	\begin{displaymath}
		\alpha\sim\begin{pmatrix}
			x\\x'
		\end{pmatrix}\sim
		\begin{pmatrix}
			z\\z'
		\end{pmatrix}\sim
		\begin{pmatrix}
			y\\y'
		\end{pmatrix}\sim\beta,
	\end{displaymath}
	which implies that there is a path from $\alpha$ to $\beta$ in $\commgraph{\ptr{X}}$ whose length is at most $4$, that is, $\dist{\ptr{X}}{\alpha}{\beta}\leqslant 4$.
\end{proof}

\begin{lemma}\label{P(X): upper bound diam, a€T(X)-S(X) b€P(X)-(T(X) U 0)}
	Suppose that $\abs{X}\geqslant 4$. Let $\alpha\in\tr{X}\setminus\sym{X}$ and $\beta\in\ptr{X}\setminus\parens{\tr{X}\cup\set{\emptyset}}$. Then $\dist{\ptr{X}}{\alpha}{\beta}\leqslant 4$.
\end{lemma}

\begin{proof}
	It follows from the fact that $\beta\notin\tr{X}$ that $X\setminus\dom\beta\neq\emptyset$ and $X\setminus\im\beta\neq\emptyset$. Let $x'\in X\setminus\dom\beta$ and $x\in X\setminus\im\beta$. We have that
	\begin{displaymath}
		\beta\sim\begin{pmatrix}
			x\\x'
		\end{pmatrix}
	\end{displaymath}
	since the relevant products are equal to the empty map. We divide the rest of the proof into 3 cases.
	
	\smallskip
	
	\textit{Case 1:} Assume that $\alpha$ is an idempotent of rank 1. Let $y\in X$ be such that $\im\alpha=\set{y}$ and let $y'\in X\setminus\set{x,x',y}$ (note that such an element exists because $\abs{X}\geqslant 4$). Let $\gamma\in\tr{X}$ be the transformation defined by
	\begin{displaymath}
		z\gamma=\begin{cases}
			z& \text{if } z\in\set{x,x',y},\\
			y'& \text{if } z\in X\setminus\set{x,x',y}
		\end{cases}
	\end{displaymath}
	for all $z\in X$. Then we have
	\begin{align*}
		\parens{z\alpha}\gamma&=y\gamma &\bracks{\text{because } \im\alpha=\set{y}}\\
		&=y\\
		&=\parens{z\gamma}\alpha &\bracks{\text{because } \im\alpha=\set{y}}
	\end{align*}
	for all $z\in X$; that is, $\alpha\gamma=\gamma\alpha$. Additionally, the fact that $\set{x}\gamma^{-1}=\set{x}$ and $x'\gamma=x'$ implies that
	\begin{displaymath}
		\gamma\begin{pmatrix}
			x\\x'
		\end{pmatrix}=\begin{pmatrix}
			x\\x'
		\end{pmatrix}=\begin{pmatrix}
			x\\x'
		\end{pmatrix}\gamma.
	\end{displaymath}
	Therefore we have
	\begin{displaymath}
		\alpha\sim\gamma\sim\begin{pmatrix}
			x\\x'
		\end{pmatrix}\sim\beta
	\end{displaymath}
	and, consequently, we can conclude that there is a path from $\alpha$ to $\beta$ in $\commgraph{\ptr{X}}$ whose length is at most $3$. Thus $\dist{\commgraph{\ptr{X}}}{\alpha}{\beta}\leqslant 3$.
	
	\smallskip
	
	\textit{Case 2:} Assume that $\alpha$ is an idempotent of rank at least $2$. In the next two sub-cases we will see that $\dist{\commgraph{\ptr{X}}}{\alpha}{\beta}\leqslant 3$.
	
	\smallskip
	
	\textsc{Sub-case 1:} Assume that there exists $y\in\im\alpha$ such that $x,x'\in X\setminus \parens{\set{y}\alpha^{-1}}$ (that is, such that $x\alpha\neq y$ and $x'\alpha\neq y$). Let $\gamma\in\ptr{X}\setminus\tr{X}$ be such that $\dom\gamma=\set{y}\alpha^{-1}$ and $\im\gamma=\set{y}$. We have that
	\begin{align*}
		\dom\alpha\gamma &=\parens{\im\alpha\cap\dom\gamma}\alpha^{-1}\\
		&=\parens{\im\alpha\cap\parens{\set{y}\alpha^{-1}}}\alpha^{-1} &\bracks{\text{because } \dom\gamma=\set{y}\alpha^{-1}}\\
		&=\set{y}\alpha^{-1} &\kern -0.3cm \bracks{\text{because } \alpha\in\idemp{\tr{X}} \text{ and so } z\alpha=z \text{ for all } z\in\im\alpha}\\
		&=\dom\gamma\\
		&=\parens{\im\gamma}\gamma^{-1}\\
		&=\parens{\im\gamma\cap X}\gamma^{-1}\\
		&=\parens{\im\gamma\cap\dom\alpha}\gamma^{-1} &\bracks{\text{because } \alpha\in\tr{X}}\\
		&=\dom\gamma\alpha
	\end{align*}
	and for all $z\in\dom\alpha\gamma=\dom\gamma\alpha=\set{y}\alpha^{-1}$ we have
	\begin{align*}
		\parens{z\alpha}\gamma &=y &\bracks{\text{because } \im\gamma=\set{y}}\\
		&=y\alpha &\bracks{\text{because } \alpha\in\idemp{\tr{X}} \text{ and so } y\alpha=y}\\
		&=\parens{z\gamma}\alpha. &\bracks{\text{because } \im\gamma=\set{y}}
	\end{align*}
	Consequently, $\alpha\gamma=\gamma\alpha$. Additionally, the fact that $x,x'\in X\setminus\parens{\set{y}\alpha^{-1}}$ ensures that $x\in X\setminus\set{y}=X\setminus\im\gamma$ (because $y\alpha=y$) and $x'\in X\setminus\dom\gamma$, which implies that
	\begin{displaymath}
		\gamma\begin{pmatrix}
			x\\x'
		\end{pmatrix}=\emptyset=\begin{pmatrix}
			x\\x'
		\end{pmatrix}\gamma.
	\end{displaymath}
	Thus we have
	\begin{displaymath}
		\alpha\sim\gamma\sim\begin{pmatrix}
			x\\x'
		\end{pmatrix}\sim\beta,
	\end{displaymath}
	which allows us to conclude that there is a path from $\alpha$ to $\beta$ in $\commgraph{\ptr{X}}$ whose length is at most $3$. Therefore $\dist{\commgraph{\tr{X}}}{\alpha}{\beta}\leqslant 3$.
	
	\smallskip
	
	\textsc{Sub-case 2:} Assume that for all $y\in\im \alpha$ we have $x\in \set{y}\alpha^{-1}$ or $x'\in \set{y}\alpha^{-1}$ (that is, $x\alpha=y$ or $x'\alpha=y$). This implies that $\abs{\im\alpha}\leqslant 2$ and, since we also have $\abs{\im\alpha}\geqslant 2$ (recall that the rank of $\alpha$ is at least $2$), then we must have $\abs{\im\alpha}=2$. Assume that $\im\alpha=\set{y,y'}$. (We observe that, since $\alpha$ is an idempotent, then we have $y\alpha=y$ and $y'\alpha=y'$.) We have that one of $x,x'$ is in $\set{y}\alpha^{-1}$ and the other one is in $\set{y'}\alpha^{-1}$. We can assume, without loss of generality, that $x\in \set{y}\alpha^{-1}$ and $x'\in \set{y'}\alpha^{-1}$ (that is, $x\alpha=y$ and $x'\alpha=y'$). Let $\gamma\in\ptr{X}$ be such that $\dom\gamma=\set{y}\alpha^{-1}$ and $\im\gamma=\set{y'}$. We have that
	\begin{align*}
		\dom\alpha\gamma &=\parens{\im\alpha\cap\dom\gamma}\alpha^{-1}\\
		&=\parens{\set{y,y'}\cap \parens{\set{y}\alpha^{-1}}}\alpha^{-1} &\bracks{\text{because } \im\alpha=\set{y,y'} \text{ and } \dom\gamma=\set{y}\alpha^{-1}}\\
		&=\set{y}\alpha^{-1} &\bracks{\text{because } y\alpha=y \text{ and } y'\alpha=y'\neq y}\\
		&=\dom\gamma\\
		&=\parens{\im\gamma}\gamma^{-1}\\
		&=\parens{\im\gamma\cap X}\gamma^{-1}\\
		&=\parens{\im\gamma\cap\dom\alpha}\gamma^{-1} &\bracks{\text{because } \alpha\in\tr{X}}\\
		&=\dom\gamma\alpha
	\end{align*}
	and for all $z\in\dom\alpha\gamma=\dom\gamma\alpha=\set{y}\alpha^{-1}$ we have that
	\begin{align*}
		\parens{z\alpha}\gamma&=y' &\bracks{\text{because } \im\gamma=\set{y'}}\\
		&=y'\alpha\\
		&=\parens{z\gamma}\alpha. &\bracks{\text{because } \im\gamma=\set{y'}}
	\end{align*}
	Hence $\alpha\gamma=\gamma\alpha$. Moreover, $x\alpha=y\neq y'=y'\alpha$ and $x'\alpha=y'\neq y$, which implies that $x\in X\setminus\set{y'}=X\setminus\im\gamma$ and $x'\in X\setminus \parens{\set{y}\alpha^{-1}}=X\setminus\dom\gamma$. Consequently, we have
	\begin{displaymath}
		\gamma\begin{pmatrix}
			x\\x'
		\end{pmatrix}=\emptyset=\begin{pmatrix}
			x\\x'
		\end{pmatrix}\gamma.
	\end{displaymath}
	Therefore we have
	\begin{displaymath}
		\alpha\sim\gamma\sim\begin{pmatrix}
			x\\x'
		\end{pmatrix}\sim\beta,
	\end{displaymath}
	which implies that there is a path from $\alpha$ to $\beta$ in $\commgraph{\ptr{X}}$ whose length is at most $3$. Thus $\dist{\commgraph{\ptr{X}}}{\alpha}{\beta}\leqslant 3$.
	
	\smallskip
	
	\textit{Case 3:} Assume that $\alpha$ is not an idempotent. There exists $m\in\mathbb{N}$ such that $\alpha^m$ is an idempotent. Furthermore, we have that $\alpha\notin\sym{X}$, which implies that $\alpha^m\neq\id{X}$. Hence $\alpha^m$ is a vertex of $\commgraph{\ptr{X}}$ and $\alpha-\alpha^{m}$ is a path in $\commgraph{\ptr{X}}$. Consequently, $\dist{\commgraph{\ptr{X}}}{\alpha}{\alpha^m}=1$. By cases 1 and 2 we also know that $\dist{\commgraph{\ptr{X}}}{\alpha^m}{\beta}\leqslant 3$. Thus $\dist{\commgraph{\ptr{X}}}{\alpha}{\beta}\leqslant 4$.
\end{proof}

\begin{lemma}\label{P(X): upper bound diam, a€S(X)-1 b€P(X)-(T(X) U 0)}
	Suppose that $\abs{X}\geqslant 4$ and is composite. Let $\alpha\in\sym{X}\setminus\set{\id{X}}$ and $\beta\in\ptr{X}\setminus\parens{\tr{X}\cup\set{\emptyset}}$. Then $\dist{\ptr{X}}{\alpha}{\beta}\leqslant 5$. Moreover, if $\abs{X}=4$, then $\dist{\commgraph{\ptr{X}}}{\alpha}{\beta}\leqslant 4$.
\end{lemma}

\begin{proof}
	
	We divide this proof into two parts. In the first one we will see that $\dist{\commgraph{\ptr{X}}}{\alpha}{\beta}\leqslant 5$; and in the second one we will see that, if $\abs{X}=4$, then $\dist{\commgraph{\ptr{X}}}{\alpha}{\beta}\leqslant 4$.
	
	Before we start part 1 we note that, due to the fact that $\beta\in\ptr{X}\setminus\tr{X}$, there exist $x\in X\setminus\im\beta$ and $x'\in X\setminus\dom\beta$. This implies that
	\begin{displaymath}
		\beta\sim\begin{pmatrix}
			x\\x'
		\end{pmatrix}
	\end{displaymath}
	because the relevant products are equal to $\emptyset$.
	
	\medskip
	
	\textbf{Part 1.} We are going to establish that $\dist{\commgraph{\ptr{X}}}{\alpha}{\beta}\leqslant 5$. We have that $\alpha\in\sym{X}$, which means that $\alpha$ can be written as a product of cycles. We use this information to divide this part of the proof into two cases.
	
	\smallskip
	
	\textit{Case 1:} Assume that $\alpha$ has at least two cycles. Let $\cycle{y_1 y_2 {\ldots} y_m}$ be a cycle in $\alpha$ and let $Y=\set{y_1,\ldots,y_m}\subsetneq X$. We have that
	\begin{align*}
		\dom\alpha\id{Y} &=\parens{\im\alpha \cap \dom\id{Y}}\alpha^{-1}\\
		&=\parens{X\cap Y}\alpha^{-1} &\bracks{\text{because } \alpha\in\sym{X} \text{ and } \dom\id{Y}=Y}\\
		&=Y\alpha^{-1}\\
		&=Y &\bracks{\text{because } \cycle{y_1 y_2 {\ldots} y_m} \text{ is a cycle in } \alpha}\\
		&=Y\parens{\id{Y}}^{-1}&\bracks{\text{because } \dom\id{Y}=\im\id{Y}=Y}\\
		&=\parens{Y\cap X}\parens{\id{Y}}^{-1}\\
		&=\parens{\im\id{Y}\cap\dom\alpha}\parens{\id{Y}}^{-1} &\bracks{\text{because } \alpha\in\sym{X} \text{ and } \im\id{Y}=Y}\\
		&=\dom\id{Y}\alpha
	\end{align*}
	and $\parens{z\alpha}\id{Y}= z\alpha= \parens{z\id{Y}}\alpha$ for all $z\in\dom\alpha\id{Y}=\dom\id{Y}\alpha$. Hence $\alpha\id{Y}=\id{Y}\alpha$. Let $y\in X\setminus\set{x,x'}$. Since $\id{Y},\id{\set{y}}\in\idemp{\psym{X}}$ (a commutative subsemigroup of $\ptr{X}$) we have that $\id{Y}\id{\set{y}}=\id{\set{y}}\id{Y}$ and, since $y\in X\setminus\set{x,x'}$, we have that
	\begin{displaymath}
		\id{\set{y}}\begin{pmatrix}
			x\\x'
		\end{pmatrix}=\emptyset=\begin{pmatrix}
			x\\x'
		\end{pmatrix}\id{\set{y}}.
	\end{displaymath}
	Therefore we have
	\begin{displaymath}
		\alpha\sim\id{Y}\sim\id{\set{y}}\sim\begin{pmatrix}
			x\\x'
		\end{pmatrix}\sim\beta
	\end{displaymath}
	and, consequently, $\dist{\commgraph{\tr{X}}}{\alpha}{\beta}\leqslant 4$.
	
	\smallskip
	
	\textit{Case 2:} Assume that $\alpha$ is a cycle (of length $\abs{X}$). Since $\abs{X}$ is composite, then there exists $m\in\set{2,\ldots,\abs{X}-1}$ such that $m$ divides $\abs{X}$. Hence $\alpha^m$ is given by the product of $m$ cycles, all of which have length $\frac{\abs{X}}{m}$. We know that $\alpha-\alpha^m$ is a path in $\commgraph{\ptr{X}}$, which implies that $\dist{\commgraph{\ptr{X}}}{\alpha}{\alpha^m}=1$. Moreover, it follows from case 1, and the fact that $\alpha^m\in\sym{X}$ has more than one cycle, that $\dist{\commgraph{\ptr{X}}}{\alpha^m}{\beta}\leqslant 4$. Thus $\dist{\commgraph{\ptr{X}}}{\alpha}{\beta}\leqslant 5$.
	
	\medskip
	
	\textbf{Part 2.} Suppose that $\abs{X}=4$. We are going to show that $\dist{\commgraph{\ptr{X}}}{\alpha}{\beta}\leqslant 4$. In case 1 of part 1 we proved that, if $\alpha$ has at least two cycles, then $\dist{\commgraph{\ptr{X}}}{\alpha}{\beta}\leqslant 4$. Therefore we only need to check that, if $\alpha$ is a cycle (of length $\abs{X}=4$), then $\dist{\commgraph{\ptr{X}}}{\alpha}{\beta}\leqslant 4$.
	
	Assume that $\alpha$ is a cycle (of length $\abs{X}=4$) and that $\alpha=\cycle{x_1 x_2 x_3 x_4}$. Then $\alpha^2=\cycle{x_1 x_3}\cycle{x_2 x_4}$. We note that $\alpha^2$ commutes with $\alpha$. We have two possible cases.
	
	\smallskip
	
	\textit{Case 1:} Assume that either $x,x'\in \set{x_1,x_3}$ or $x,x'\in\set{x_2,x_4}$. We can assume, without loss of generality, that $x,x'\in\set{x_1,x_3}$. Since $\cycle{x_2 x_4}$ is a cycle in $\alpha^2$, then we have that $\alpha^2\id{\set{x_2,x_4}}=\id{\set{x_2,x_4}}\alpha^2$ (this can be proved in a similar way to case 1 of part 1). Moreover, we have that $x,x'\in\set{x_1,x_3}=X\setminus\set{x_2,x_4}$, which implies that
	\begin{displaymath}
		\id{\set{x_2,x_4}}\begin{pmatrix}
			x\\x'
		\end{pmatrix}=\emptyset=\begin{pmatrix}
			x\\x'
		\end{pmatrix}\id{\set{x_2,x_4}}.
	\end{displaymath}
	Therefore we have
	\begin{displaymath}
		\alpha\sim\alpha^2\sim\id{\set{x_2,x_4}}\sim\begin{pmatrix}
			x\\x'
		\end{pmatrix}\sim\beta
	\end{displaymath}
	and, consequently, $\dist{\commgraph{\ptr{X}}}{\alpha}{\beta}\leqslant 4$.
	
	\smallskip
	
	\textit{Case 2:} Assume that either $x\in\set{x_1,x_3}$ and $x'\in\set{x_2,x_4}$, or $x\in\set{x_2,x_4}$ and $x'\in\set{x_1,x_3}$. We assume, without loss of generality, that $x\in\set{x_1,x_3}$ and $x'\in\set{x_2,x_4}$. We have that
	\begin{multline*}
		\alpha^2\begin{pmatrix}
			x_1&x_3\\
			x_4&x_2
		\end{pmatrix}
		=\begin{pmatrix}
			x_1&x_2&x_3&x_4\\
			x_3&x_4&x_1&x_2
		\end{pmatrix}\begin{pmatrix}
			x_1&x_3\\
			x_4&x_2
		\end{pmatrix}
		=\begin{pmatrix}
			x_1&x_3\\
			x_2&x_4
		\end{pmatrix}\\
		=\begin{pmatrix}
			x_1&x_3\\
			x_4&x_2
		\end{pmatrix}\begin{pmatrix}
			x_1&x_2&x_3&x_4\\
			x_3&x_4&x_1&x_2
		\end{pmatrix}
		=\begin{pmatrix}
			x_1&x_3\\
			x_4&x_2
		\end{pmatrix}\alpha^2.
	\end{multline*}
	and, since $x\in\set{x_1,x_3}=X\setminus\set{x_2,x_4}$ and $x'\in\set{x_2,x_4}=X\setminus\set{x_1,x_3}$, we have that
	\begin{displaymath}
		\begin{pmatrix}
			x_1&x_3\\
			x_4&x_2
		\end{pmatrix}\begin{pmatrix}
			x\\x'
		\end{pmatrix}=\emptyset=\begin{pmatrix}
			x\\x'
		\end{pmatrix}\begin{pmatrix}
			x_1&x_3\\
			x_4&x_2
		\end{pmatrix}.
	\end{displaymath}
	Therefore we have
	\begin{displaymath}
		\alpha\sim\alpha^2\sim\begin{pmatrix}
			x_1&x_3\\
			x_4&x_2
		\end{pmatrix}\sim\begin{pmatrix}
			x\\x'
		\end{pmatrix}\sim\beta,
	\end{displaymath}
	which allows us to conclude that $\dist{\commgraph{\ptr{X}}}{\alpha}{\beta}\leqslant 4$.
\end{proof}

With the previous lemma we complete the first part of determining the diameter of $\commgraph{\ptr{X}}$: obtaining an upper bound for the distance between vertices of $\commgraph{\ptr{X}}$. Now we start the second part: find vertices of $\commgraph{\ptr{X}}$ such that the distance between them is $4$, if $\abs{X}=4$, and $5$, if $\abs{X}\geqslant 6$ is composite. This will be done in Lemmata~\ref{lower bound diam, |X|=4}, \ref{P(X): lower bound diam, |X|=6 composite}, \ref{P(X): lower bound diam, |X|=8 composite}~and~\ref{P(X): lower bound diam, |X|>=9 composite}. Before we proceed with those lemmata, we need a few results, which will be helpful in proving that the distance between certain vertices of $\commgraph{\ptr{X}}$ is not smaller than what we require.

We define a new graph, which is constructed from two full transformations. We will see in Lemma~\ref{P(X): unified graph lemma} that this graph will be useful in proving that, for certain $\alpha,\beta\in\tr{X}\setminus\set{\id{X}}$, there is no vertex $\gamma\in\ptr{X}\setminus\parens{\tr{X}\cup\set{\emptyset}}$ such that $\alpha-\gamma-\beta$ is a path in $\commgraph{\ptr{X}}$.

\begin{definition}
	Let $\alpha,\beta\in\tr{X}$. We define $\graphptr{\alpha}{\beta}$ as the simple graph whose vertex set is $X$ and where two distinct vertices $x,y\in X$ are adjacent if and only if $x\alpha=y$ or $y\alpha=x$ or $x\beta=y$ or $y\beta=x$.
\end{definition}

Given $\alpha,\beta\in\tr{X}$, an easy way of constructing the graph $\graphptr{\alpha}{\beta}$ is to draw all the elements of $X$ (the vertex set of $\graphptr{\alpha}{\beta}$) and then take one of the full transformations, say $\alpha$. Then for each $x\in X$ we follow the procedure described below:
\begin{enumerate}
	\item If $x\neq x\alpha$ and there is no edge between the vertices $x$ and $x\alpha$, then we draw that edge.
	
	\item If either $x=x\alpha$, or $x\neq x\alpha$ and there is already an edge between the vertices $x$ and $x\alpha$, then we take another element of $X$ that we have not considered so far.
\end{enumerate}
Once we have considered all the elements of $X$, we repeat this preceeding with the full transformation $\beta$.

In the rest of the chapter there will be several example of these graphs. See, for instance, Figures~\ref{P(X), Figure: unified graph alpha2 alpha4, e}, \ref{P(X), Figure: unified graph alpha3, e} and \ref{P(X), Figure: unified graph alpha5, e}.

\begin{lemma}\label{P(X): unified graph lemma}
	Let $\alpha,\beta\in\tr{X}$ and let $\gamma\in\ptr{X}\setminus\tr{X}$ be such that $\alpha\gamma=\gamma\alpha$ and $\beta\gamma=\gamma\beta$. If $\graphptr{\alpha}{\beta}$ is connected, then $\gamma=\emptyset$.
\end{lemma}

\begin{proof}
	
	\textbf{Part 1.} As a preliminary step in proving Lemma~\ref{P(X): unified graph lemma}, we show that, if two distinct vertices $x$ and $y$ of $\graphptr{\alpha}{\beta}$ are adjacent, then we have $x\in\dom\gamma$ if and only if $y\in\dom\gamma$.
	
	Let $x$ and $y$ be (distinct) adjacent vertices of $\graphptr{\alpha}{\beta}$. Then we have $x\alpha=y$ or $y\alpha=x$ or $x\beta=y$ or $y\beta=x$. Assume that $x\alpha=y$. (If $y\alpha=x$ or $x\beta=y$ or $y\beta=x$, then the result can be proved in an analogous way.) We have that
	\begin{align*}
		x\in\dom\gamma & \iff x\in\dom\gamma \text{ and } x\gamma\in\dom\alpha &\bracks{\text{because } x\gamma\in X=\dom\alpha}\\
		&\iff x\in\dom\gamma\alpha\\
		&\iff x\in\dom\alpha\gamma &\bracks{\text{because } \alpha\gamma=\gamma\alpha}\\
		&\iff x\in\dom\alpha \text{ and } x\alpha\in\dom\gamma\\
		&\iff x\alpha\in\dom\gamma &\bracks{\text{because } x\in X=\dom\alpha}\\
		&\iff y\in\dom\gamma. &\bracks{\text{because } y=x\alpha}
	\end{align*}
	
	\medskip
	
	\textbf{Part 2.} Now we are ready to prove Lemma~\ref{P(X): unified graph lemma}. Suppose that $\graphptr{\alpha}{\beta}$ is connected. Due to the fact that $\gamma\in\ptr{X}\setminus\tr{X}$, we know that $\dom\gamma\subsetneq X$. Hence there exists $x\in X\setminus\dom\gamma$. Let $y\in X\setminus\set{x}$. It follows from the fact that $\graphptr{\alpha}{\beta}$ is connected that there is a path from $x$ to $y$. Let
	\begin{displaymath}
		x=x_1 - x_2 - \cdots - x_n=y
	\end{displaymath}
	be one of those paths. As a consequence of the fact that $x_1=x\in X\setminus\dom\gamma$, and by an iterated use of part 1, we can successively establish that, for each $i\in\set{2,\ldots,n}$, we have $x_i\in X\setminus\dom\gamma$ and, thus, conclude that $y=x_n\in X\setminus\dom\gamma$.
	
	Since $y$ is an arbitrary element of $X\setminus\set{x}$, we can conclude that $X\setminus\set{x}\subseteq X\setminus\dom\gamma$. Additionally, we know that $x\in X\setminus\dom\gamma$. Therefore we have $\dom\gamma=\emptyset$; that is, $\gamma=\emptyset$.
\end{proof}

The next two lemmata were important in the determination of a lower bound for the diameter of $\commgraph{\tr{X}}$ \cite[Theorem 2.22]{Commuting_graph_T_X}. We will also need them to determine a lower bound for the diameter of $\commgraph{\ptr{X}}$. Before we present those lemmata, we introduce two notations that were used in \cite[Subsection 2.1]{Commuting_graph_T_X} for denoting certain types of full transformations.

First, we introduce a notation for idempotents of $\tr{X}$: if $\set{A_i}_{i=1}^n$ is a partition of $X$, where $n\in\mathbb{N}$, and $x_i\in A_i$ for all $i\in\Xn$, then we denote by
\begin{displaymath}
	e=\chain{A_1, x_1}\chain{A_2, x_2}\cdots\chain{A_n, x_n}
\end{displaymath}
the idempotent such that $A_ie=\set{x_i}$ for all $i\in\Xn$.

We note that all idempotents $e$ of $\tr{X}$ can be written using that notation. In fact, if $\im e=\set{x_1,\ldots,x_n}$ and $A_i=\set{x_i}e^{-1}$ for all $i\in\Xn$, then $\set{A_i}_{i=1}^n$ is a partition of $X$ and, since $e$ is an idempotent, we have $x_i\in \set{x_i}e^{-1}=A_i$ for all $i\in\Xn$. Then we can write $e$ using the notation introduced above.

Now we introduce a notation for particular full transformations. Let $m,k\in\mathbb{N}$ and assume that $X=\set{x_1,\ldots,x_m,y_1,\ldots,y_k}$ (where $\abs{X}=m+k$). We denote by
\begin{displaymath}
	\alpha=\chain{y_k y_{k-1} {\ldots} y_1 x_1}\cycle{x_1 x_2 {\ldots} x_m}
\end{displaymath}
the full transformation defined as follows: for all $i\in\X{m}$ and $j\in\X{k}$ we have
\begin{displaymath}
	x_i\alpha=\begin{cases}
		x_{i+1} &\text{if } i<m,\\
		x_1 &\text{if } i=m;
	\end{cases} \quad \text{and} \quad
	y_j\alpha=\begin{cases}
		x_1 &\text{if } j=1,\\
		y_{j-1} &\text{if } j>1.
	\end{cases}
\end{displaymath}
So this means that the segment $\cycle{x_1 x_2 {\ldots} x_m}$ can be read the same way as we would read a cycle of a permutation, and the segment $\chain{y_k y_{k-1} {\ldots} y_1 x_1}$ can almost be read as a cycle of a permutation: the only exception is $x_1$, whose image is not defined in $\chain{y_k y_{k-1} {\ldots} y_1 x_1}$.

\begin{lemma}[{\cite[Lemma 2.12]{Commuting_graph_T_X}}]\label{P(X): idempotent that commutes with chain+cycle}
	Suppose that $X=\set{x_1,\ldots,x_m,y_1,\ldots,y_k}$, where $m,k\in\mathbb{N}$. Let $\alpha=\chain{y_k y_{k-1} {\ldots} y_1 x_1}\cycle{x_1 x_2 {\ldots} x_m}$ and let $e\in\tr{X}$ be an idempotent such that $\alpha e=e\alpha$. If $e\neq\id{X}$, then
	\begin{enumerate}
		\item For all $i\in\X{m}$ we have $x_ie=x_i$.
		\item For all $i\in\X{k}$ we have $y_ie=x_{i^*}$, where $i^*\in\X{m}$ is such that $i^* \equiv m-i+1 \pmod{m}$.
	\end{enumerate}
\end{lemma}

\begin{lemma}[{\cite[Lemma 2.20]{Commuting_graph_T_X}}]\label{P(X): upper bound, chain + cycle must commute with id-->S(X)}
	Suppose that $X=\set{x_1,\ldots,x_m,y_1,\ldots,y_k}$, where $m,k\in\mathbb{N}$. Let $\alpha=\chain{y_1 {\ldots} y_k x_1}\cycle{x_1 {\ldots} x_m}$ and $\beta\in\sym{X}$. If $\alpha\beta=\beta\alpha$, then $\beta=\id{X}$.
\end{lemma}

The previous lemma shows that the identity is the only permutation that commutes with the full transformation $\chain{y_1 {\ldots} y_k x_1}\cycle{x_1 {\ldots} x_m}$. The next lemma proves something similar: it establishes that the empty map is the only partial transformation that is not a full transformation and that commutes with $\chain{y_1 {\ldots} y_k x_1}\cycle{x_1 {\ldots} x_m}$.

\begin{lemma}\label{P(X): upper bound, chain + cycle must commute with empty map-->P(X)-T(X)}
	Suppose that $X=\set{x_1,\ldots,x_m,y_1,\ldots,y_k}$, where $m,k\in\mathbb{N}$. Let $\alpha=\chain{y_1 {\ldots} y_k x_1}\cycle{x_1 {\ldots} x_m}$ and $\beta\in\ptr{X}\setminus\tr{X}$. If $\alpha\beta=\beta\alpha$, then $\beta=\emptyset$.
\end{lemma}

\begin{proof}
	Suppose that $\alpha\beta=\beta\alpha$. In Figure~\ref{P(X), Figure: unified graph chain, chain} we have the graph $\graphptr{\alpha}{\alpha}$. It is clear, by observing Figure~\ref{P(X), Figure: unified graph chain, chain}, that $\graphptr{\alpha}{\alpha}$ is connected. Hence, by lemma~\ref{P(X): unified graph lemma}, we have that $\beta=\emptyset$.
	\begin{figure}[hbt]
		\begin{center}
			\begin{tikzpicture}
				\foreach \x in {0,-1,-2,-4} {
					\node[vertex] (x\x) at (180+\x*360/5:1cm) {};
				}
				
				\node[vertex] (y1) at (-6,0) {};
				\node[vertex] (y2) at (-5,0) {};
				\node[vertex] (y3) at (-4,0) {};
				\node[vertex] (yk) at (-2,0) {};
				
				\draw (168:1.32cm) node {$x_1$};
				\draw (180-1*360/5:1.4cm) node {$x_2$};
				\draw (180-2*360/5:1.4cm) node {$x_3$};
				\draw (180-4*360/5:1.4cm) node {$x_n$};
				
				\node[anchor=north,inner sep=2mm] at (y1) {$y_1$};
				\node[anchor=north,inner sep=2mm] at (y2) {$y_2$};
				\node[anchor=north,inner sep=2mm] at (y3) {$y_3$};
				\node[anchor=north,inner sep=2mm] at (yk) {$y_k$};
				
				\begin{scope}[edge]
					\draw (180:1cm) arc (180:180-1*360/5:1cm);
					\draw (180-1*360/5:1cm) arc (180-1*360/5:180-2*360/5:1cm);
					\draw (180-4*360/5:1cm) arc (180-4*360/5:180-5*360/5:1cm);
					\draw[dotted] (180-2*360/5:1cm) arc (180-2*360/5:180-4*360/5:1cm);
					
					\draw (y1) -- (y2);
					\draw (y2) -- (y3);
					\draw[dotted] (y3) -- (yk);
					\draw (yk) -- (x0);
				\end{scope}
			\end{tikzpicture}
		\end{center}
		\caption{Graph $\graphptr{\alpha}{\alpha}$ where $\alpha=\chain{y_1 {\ldots} y_k x_1}\cycle{x_1 {\ldots} x_m}$.}
		\label{P(X), Figure: unified graph chain, chain}
	\end{figure}
\end{proof}

Now we describe the partial transformations that commute with a cycle of length $\abs{X}$.

\begin{lemma}\label{P(X): what commutes with cycles length |X|}
	Let $\alpha\in\sym{X}$ and let $\beta\in\ptr{X}$ be such that $\alpha\beta=\beta\alpha$. If $\alpha$ is a cycle of length $\abs{X}$, then $\beta=\emptyset$ or $\beta$ is some power of $\alpha$.
\end{lemma}

\begin{proof}
	Let $n=\abs{X}$ and assume that $X=\set{x_1,\ldots,x_n}$ and $\alpha=\cycle{x_1 x_2 {\ldots} x_n}$.
	
	\smallskip
	
	\textit{Case 1:} Suppose that $\beta\in\ptr{X}\setminus\tr{X}$. Figure~\ref{P(X), Figure: unified graph cycle, cycle} shows the graph $\graphptr{\alpha}{\alpha}$. We can easily conclude that $\graphptr{\alpha}{\alpha}$ is connected and, thus, Lemma~\ref{P(X): unified graph lemma} implies that $\beta=\emptyset$.
	
	\begin{figure}[hbt]
		\begin{center}
			\begin{tikzpicture}
				\foreach \x in {0,-1,-2,-4} {
					\node[vertex] at (90+\x*360/5:1cm) {};
				}
				
				\draw (90:1.4cm) node {$x_1$};
				\draw (90-1*360/5:1.4cm) node {$x_2$};
				\draw (90-2*360/5:1.4cm) node {$x_3$};
				\draw (90-4*360/5:1.4cm) node {$x_n$};
				
				\begin{scope}[edge]
					\draw (90:1cm) arc (90:90-1*360/5:1cm);
					\draw (90-1*360/5:1cm) arc (90-1*360/5:90-2*360/5:1cm);
					\draw (90-4*360/5:1cm) arc (90-4*360/5:90-5*360/5:1cm);
					\draw[dotted] (90-2*360/5:1cm) arc (90-2*360/5:90-4*360/5:1cm);
				\end{scope}
				
			\end{tikzpicture}
		\end{center}
		\caption{Graph $\graphptr{\alpha}{\alpha}$ where $\alpha=\cycle{x_1 x_2 {\ldots} x_n}$.}
		\label{P(X), Figure: unified graph cycle, cycle}
	\end{figure}
	
	\smallskip
	
	\textit{Case 2:} Suppose that $\beta\in\tr{X}$. Let $i\in\Xn$ be such that $x_1\beta=x_i$. For all $j\in\Xn$ have that
	\begin{displaymath}
		x_j\beta=x_1\alpha^{j-1}\beta=x_1\beta\alpha^{j-1}=x_i\alpha^{j-1}=x_1\alpha^{i-1}\alpha^{j-1}=x_1\alpha^{j-1}\alpha^{i-1}=x_j\alpha^{i-1}.
	\end{displaymath}
	Thus $\beta=\alpha^{i-1}$, a power of $\alpha$.
\end{proof}

Now we can determine a lower bound for the diameter of $\commgraph{\ptr{X}}$. Lemmata~\ref{lower bound diam, |X|=4}, \ref{P(X): lower bound diam, |X|=6 composite}, \ref{P(X): lower bound diam, |X|=8 composite} and \ref{P(X): lower bound diam, |X|>=9 composite} provide those lower bounds when $\abs{X}=4,6,8$ and $\abs{X}\geqslant 9$ is composite, respectively.

\begin{lemma}\label{lower bound diam, |X|=4}
	Suppose that $\abs{X}=4$. Then there exist $\alpha,\beta\in\ptr{X}\setminus\set{\emptyset,\id{X}}$ such that $\dist{\commgraph{\ptr{X}}}{\alpha}{\beta}\geqslant 4$.
\end{lemma}

\begin{proof}
	Assume that $X=\set{1,2,3,4}$. Let
	\begin{displaymath}
		\alpha=\cycle{1 2 3 4} \quad \text{and} \quad \beta=\chain{1 2 3}\cycle{3 4}.
	\end{displaymath}
	and let
	\begin{displaymath}
		\alpha=\gamma_1-\gamma_2-\cdots-\gamma_n=\beta
	\end{displaymath}
	be a path of minimum length from $\alpha$ to $\beta$ (in $\commgraph{\ptr{X}}$). Then $\dist{\commgraph{\ptr{X}}}{\alpha}{\beta}=n-1$.
	
	
	We have that $3\alpha\beta=4\beta=3\neq 1=4\alpha=3\beta\alpha$, which implies that the vertices $\alpha$ and $\beta$ are not adjacent. Hence $\dist{\commgraph{\ptr{X}}}{\alpha}{\beta}\geqslant 2$ (and $n\geqslant 3$).
	
	It follows from Lemma~\ref{P(X): what commutes with cycles length |X|} that the partial transformations that commute with $\alpha$ are precisely $\emptyset,\id{X},\alpha,\alpha^2,\alpha^3$. Since $\centre{\ptr{X}}=\set{\emptyset,\id{X}}$, then we can conclude that the vertex $\alpha$ is only adjacent to two vertices: $\alpha^2$ and $\alpha^3$. Hence $\gamma_2\in\set{\alpha^2,\alpha^3}\subseteq\sym{X}$. Moreover, we have that $\gamma_{n-1}\beta=\beta\gamma_{n-1}$ and $\gamma_{n-1}\in\ptr{X}\setminus\set{\id{X}}$. Hence, by Lemma~\ref{P(X): upper bound, chain + cycle must commute with id-->S(X)}, we must have $\gamma_{n-1}\in \ptr{X}\setminus\sym{X}$, which implies that $\gamma_2\neq\gamma_{n-1}$. Thus $n\geqslant 4$ and $\dist{\commgraph{\ptr{X}}}{\alpha}{\beta}\geqslant 3$.
	
	Since $\gamma_{n-1}\beta=\beta\gamma_{n-1}$ and $\gamma_{n-1}\in\ptr{X}\setminus\set{\emptyset}$, then, by Lemma~\ref{P(X): upper bound, chain + cycle must commute with empty map-->P(X)-T(X)}, we must have $\gamma_{n-1}\in\tr{X}$. Hence $\gamma_{n-1}\in \tr{X}\cap\parens{\ptr{X}\setminus\sym{X}}=\tr{X}\setminus\sym{X}$. We know that there exists $m\in\mathbb{N}$ such that $\gamma_{n-1}^m$ is an idempotent. Note that we have $\gamma_{n-1}^m\neq\id{X}$ because $\gamma_{n-1}\in\tr{X}\setminus\sym{X}$. Moreover, we have $\gamma_{n-1}^m\beta=\beta\gamma_{n-1}^m$ because $\gamma_{n-1}\beta=\beta\gamma_{n-1}$. Consequently, by Lemma~\ref{P(X): idempotent that commutes with chain+cycle}, we have that
	\begin{displaymath}
		\gamma_{n-1}^m=\begin{pmatrix}
			1&2&3&4\\
			3&4&3&4
		\end{pmatrix}.
	\end{displaymath}
	Then $3\alpha^2\gamma_{n-1}^m=1\gamma_{n-1}^m=3 \neq 1=3\alpha^2=3\gamma_{n-1}^m\alpha^2$ and $3\alpha^3\gamma_{n-1}^m=2\gamma_{n-1}^m=4 \neq 2=3\alpha^3=3\gamma_{n-1}^m\alpha^3$. Hence $\gamma_2$ (which is either $\alpha^2$ or $\alpha^3$) does not commute with $\gamma_{n-1}^m$ and, consequently, $\gamma_2$ does not commute with $\gamma_{n-1}$. This implies that the vertices $\gamma_2$ and $\gamma_{n-1}$ are not adjacent. Thus we must have $n\geqslant 5$ and $\dist{\commgraph{\ptr{X}}}{\alpha}{\beta}\geqslant 4$, which concludes the proof.
\end{proof}

\begin{lemma}\label{P(X): lower bound diam, |X|=6 composite}
	Suppose that $\abs{X}=6$. Then there exist $\alpha,\beta\in\ptr{X}\setminus\set{\emptyset,\id{X}}$ such that $\dist{\commgraph{\ptr{X}}}{\alpha}{\beta}\geqslant 5$.
\end{lemma}

\begin{proof}
	Assume that $X=\set{1,2,3,4,5,6}$ and let
	\begin{displaymath}
		\alpha=\cycle{1 2 3 4 5 6} \quad \text{and} \quad \beta=\chain{6 4 1 2}\cycle{2 3 5}.
	\end{displaymath}
	In \cite[Theorem 2.22]{Commuting_graph_T_X} the authors mentioned that they verified computationally that $\dist{\commgraph{\tr{X}}}{\alpha}{\beta}\geqslant 5$. Our goal is to see that we also have $\dist{\commgraph{\ptr{X}}}{\alpha}{\beta}\geqslant 5$.
	
	Let
	\begin{displaymath}
		\alpha=\gamma_1-\gamma_2-\cdots-\gamma_n=\beta
	\end{displaymath}
	be a path of minimum length from $\alpha$ to $\beta$ (in $\commgraph{\ptr{X}}$). Since this path has length $n-1$, then $\dist{\commgraph{\ptr{X}}}{\alpha}{\beta}=n-1$. We will gradually prove that $\dist{\commgraph{\ptr{X}}}{\alpha}{\beta}$ cannot be $1$, $2$, $3$ or $4$.
	
	We have $2\alpha\beta=3\beta=5\neq 4=3\alpha=2\beta\alpha$. Hence $\alpha$ and $\beta$ are not adjacent and, consequently, $\dist{\commgraph{\ptr{X}}}{\alpha}{\beta}\geqslant 2$ (and $n\geqslant 3$).
	
	It follows from Lemma~\ref{P(X): what commutes with cycles length |X|} that $\alpha$, a cycle of length $6=\abs{X}$, only commutes with elements of the set $\set{\emptyset,\id{X},\alpha,\alpha^2,\alpha^3,\alpha^4,\alpha^5}$. We have that $\gamma_2\in\ptr{X}\setminus\set{\emptyset,\id{X},\alpha}$ and $\gamma_2$ commutes with $\alpha$. Hence $\gamma_2\in\set{\alpha^2,\alpha^3,\alpha^4,\alpha^5}\subseteq\sym{X}$. Additionally, the fact that $\gamma_{n-2}\in\ptr{X}\setminus\set{\id{X}}$ and the fact that $\gamma_{n-1}$ commutes with $\beta$ imply, together with Lemma~\ref{P(X): upper bound, chain + cycle must commute with id-->S(X)}, that $\gamma_{n-1}\in\ptr{X}\setminus\sym{X}$. Thus $\gamma_2\neq\gamma_{n-1}$, which implies that $n\geqslant 4$ and $\dist{\commgraph{\ptr{X}}}{\alpha}{\beta}\geqslant 3$.
	
	We know that there exists $s\in\mathbb{N}$ such that $e=\gamma_{n-1}^s$ is an idempotent. Moreover, in the previous paragraph we proved that $\gamma_{n-1}\in\tr{X}\setminus\sym{X}$, which implies that $e=\gamma_{n-1}^s\neq\id{X}$. Moreover, we have that $e=\gamma_{n-1}^s$ commutes with $\beta$ because $\gamma_{n-1}$ commutes with $\beta$. Then we can use Lemma~\ref{P(X): idempotent that commutes with chain+cycle} to see that
	\begin{displaymath}
		e=\chain{\set{2,6}, 2}\chain{\set{3,4}, 3}\chain{\set{5,1}, 5}.
	\end{displaymath}
	
	Since $\alpha=\gamma_1-\gamma_2-\cdots-\gamma_n=\beta$ is a path from $\alpha$ to $\beta$ of minimum length, $e$ cannot be any of the elements $\gamma_1,\ldots,\gamma_{n-2}$ (because, otherwise, we would obtain a shorter path from $\alpha$ to $\beta$). Thus
	\begin{displaymath}
		\alpha=\gamma_1-\gamma_2-\cdots-\gamma_{n-2}-e-\gamma_n=\beta
	\end{displaymath}
	is another path of minimum length from $\alpha$ to $\beta$ (in $\commgraph{\ptr{X}}$). (Note that $e=\gamma_{n-1}^s$ commutes with $\gamma_{n-2}$ and $\beta$ because $\gamma_{n-1}$ commutes with $\gamma_{n-1}$ and $\beta$.)
	
	Since $\dist{\commgraph{\tr{X}}}{\alpha}{\beta}\geqslant 5$, then there are no paths in $\commgraph{\tr{X}}$ from $\alpha$ to $\beta$ of length $3$, which implies that there are no paths in $\commgraph{\ptr{X}}$ from $\alpha$ to $\beta$ whose vertices are all full transformations and whose length is $3$. Hence $\alpha-\gamma_2-e-\beta$ cannot be a path in $\commgraph{\ptr{X}}$ (we note that $\alpha,\gamma_2,e,\beta\in\tr{X}\setminus\set{\id{X}}$). Thus we must have $n\geqslant 5$ and $\dist{\commgraph{\ptr{X}}}{\alpha}{\beta}\geqslant 4$.
	
	In order to finish the proof we just need to establish that $\dist{\commgraph{\ptr{X}}}{\alpha}{\beta}\geqslant 5$. This can be done be verifying that there is no vertex that is simultaneously adjacent to $\gamma_2$ and $e$. Let $\gamma\in\ptr{X}$ be such that $\gamma_2\gamma=\gamma\gamma_2$ and $e\gamma=\gamma e$. We have $\gamma\in\ptr{X}\setminus\tr{X}$ or $\gamma\in\tr{X}$, which motivates the following two cases.
	
	\smallskip
	
	\textit{Case 1:} Suppose that $\gamma\in\ptr{X}\setminus\tr{X}$. We want to see that $\gamma=\emptyset$. By Lemma~\ref{P(X): unified graph lemma} it is enough to see that $\graphptr{\gamma_2}{e}$ is connected. Moreover, we know that $\gamma_2\in\set{\alpha^2,\alpha^3,\alpha^4,\alpha^5}$. In Figure~\ref{P(X), Figure: unified graph alpha2 alpha4, e} we have the graph $\graphptr{\alpha^2}{e}=\graphptr{\alpha^4}{e}$, in Figure~\ref{P(X), Figure: unified graph alpha3, e} we have the graph $\graphptr{\alpha^3}{e}$, and in Figure~\ref{P(X), Figure: unified graph alpha5, e} we have the graph $\graphptr{\alpha^5}{e}$. By observing these three figures we can conclude that $\graphptr{\alpha^2}{e}$, $\graphptr{\alpha^3}{e}$, $\graphptr{\alpha^4}{e}$ and $\graphptr{\alpha^5}{e}$ are all connected, which implies that $\graphptr{\gamma_2}{e}$ is also connected (recall that $\gamma_2\in\set{\alpha^2,\alpha^3,\alpha^4,\alpha^5}$). Thus, by Lemma~\ref{P(X): unified graph lemma}, we have $\gamma=\emptyset\in\centre{\ptr{X}}$, which implies that $\gamma$ is not a vertex of $\commgraph{\ptr{X}}$.
	\begin{figure}[hbt]
		\begin{center}
			\begin{tikzpicture}
				\foreach \x in {0,-1,-2} {
					\node[vertex] (x\x) at (180+\x*360/3:0.8cm) {};
				}
				
				\foreach \x in {0,-1,-2} {
					\node[vertex] (y\x) at ($(-3,0)+(0+\x*360/3:0.8cm)$) {};
				}    
				
				\draw (180-1*360/30:1.16cm) node {$4$};
				\draw (180-1*360/3:1.2cm) node {$6$};
				\draw (180-2*360/3:1.2cm) node {$2$};
				
				\draw ($(-3,0)+(0+1*360/30:1.16cm)$) node {$3$};
				\draw ($(-3,0)+(0-1*360/3:1.2cm)$) node {$5$};
				\draw ($(-3,0)+(0-2*360/3:1.2cm)$) node {$1$};
				
				\begin{scope}[edge]
					\draw (x0) -- (y0);
					\draw (x0) -- (x-1);
					\draw (x-1) -- (x-2);
					\draw (x-2) -- (x0);
					\draw (y0) -- (y-1);
					\draw (y-1) -- (y-2);
					\draw (y-2) -- (y0);
				\end{scope}
			\end{tikzpicture}
		\end{center}
		\caption{Graph $\graphptr{\alpha^2}{e}=\graphptr{\alpha^4}{e}$ where $\alpha=\cycle{1 2 3 4 5 6}$ and $e=\chain{\set{2,6}, 2}\allowbreak\chain{\set{3,4}, 3}\allowbreak\chain{\set{5,1}, 5}$.}
		\label{P(X), Figure: unified graph alpha2 alpha4, e}
	\end{figure}
	\begin{figure}[hbt]
		\begin{center}
			\begin{tikzpicture}
				\foreach \x in {0,-1,-2,-3,-4,-5} {
					\node[vertex] (x\x) at (180+\x*360/6:1cm) {};
				}
				
				\draw (180-1*360/6:1.4cm) node {$1$};
				\draw (180-2*360/6:1.4cm) node {$5$};
				\draw (180-3*360/6:1.4cm) node {$2$};
				\draw (180-4*360/6:1.4cm) node {$6$};
				\draw (180-5*360/6:1.4cm) node {$3$};
				\draw (180-0*360/6:1.4cm) node {$4$};
				
				\begin{scope}[edge]
					\draw (0,0) circle (1cm);
				\end{scope}
			\end{tikzpicture}
		\end{center}
		\caption{Graph $\graphptr{\alpha^3}{e}$ where $\alpha=\cycle{1 2 3 4 5 6}$ and $e=\chain{\set{2,6}, 2}\allowbreak\chain{\set{3,4}, 3}\allowbreak\chain{\set{5,1}, 5}$.}
		\label{P(X), Figure: unified graph alpha3, e}
	\end{figure}
	\begin{figure}[hbt]
		\begin{center}
			\begin{tikzpicture}
				\foreach \x in {0,-1,-2,-3,-4,-5} {
					\node[vertex] (x\x) at (180+\x*360/6:1cm) {};
				}
				
				\draw (180-1*360/6:1.4cm) node {$1$};
				\draw (180-2*360/6:1.4cm) node {$6$};
				\draw (180-3*360/6:1.4cm) node {$5$};
				\draw (180-4*360/6:1.4cm) node {$4$};
				\draw (180-5*360/6:1.4cm) node {$3$};
				\draw (180-0*360/6:1.4cm) node {$2$};
				
				\begin{scope}[edge]
					\draw (0,0) circle (1cm);
					\draw (x-1) -- (x-3);
					\draw (x0) -- (x-2);
				\end{scope}
			\end{tikzpicture}
		\end{center}
		\caption{Graph $\graphptr{\alpha^5}{e}$ where $\alpha=\cycle{1 2 3 4 5 6}$ and $e=\chain{\set{2,6}, 2}\allowbreak\chain{\set{3,4}, 3}\allowbreak\chain{\set{5,1}, 5}$.}
		\label{P(X), Figure: unified graph alpha5, e}
	\end{figure}
	
	\smallskip
	
	\textit{Case 2:} Suppose that $\gamma\in\tr{X}$. We aim to see that $\gamma=\id{X}$. We have that $\dist{\commgraph{\tr{X}}}{\alpha}{\beta}\geqslant 5$. This implies that, in $\commgraph{\tr{X}}$, there are no paths from $\alpha$ to $\beta$ of length $4$. Since $\commgraph{\tr{X}}$ is a subgraph of $\commgraph{\ptr{X}}$, then, in $\commgraph{\ptr{X}}$, there are no paths from $\alpha$ to $\beta$ whose vertices are all full transformations and whose length is $4$. Hence $\alpha-\gamma_2-\gamma-e-\beta$ cannot be a path (notice that $\alpha,\gamma_2,e,\beta\in\tr{X}\setminus\set{\id{X}}$ and $\gamma\in\tr{X}$) and, consequently, we must have $\gamma=\id{X}\in\centre{\ptr{X}}$. Therefore $\gamma$ is not a vertex of $\commgraph{\ptr{X}}$.
	
	\smallskip
	
	In both cases we established that $\gamma$ is not a vertex of $\commgraph{\ptr{X}}$. Since $\gamma$ is an arbitrary element of $\ptr{X}$ that commutes with $\gamma_2$ and $e$, then we can conclude that $\commgraph{\ptr{X}}$ has no vertex that is adjacent to both $\gamma_2$ and $e$. This implies that $\gamma_3\neq\gamma_{n-2}$. Hence $n\geqslant 6$ and $\dist{\commgraph{\ptr{X}}}{\alpha}{\beta}\geqslant 5$.
\end{proof}

\begin{lemma}\label{P(X): lower bound diam, |X|=8 composite}
	Suppose that $\abs{X}=8$. Then there exist $\alpha,\beta\in\ptr{X}\setminus\set{\emptyset,\id{X}}$ such that $\dist{\commgraph{\ptr{X}}}{\alpha}{\beta}\geqslant 5$.
\end{lemma}

\begin{proof}
	Assume that $X=\set{1,2,3,4,5,6,7,8}$ and let
	\begin{displaymath}
		\alpha=\cycle{1 2 3 4 5 6 7 8} \quad \text{and} \quad \beta=\chain{7 6 8 5 4 1}\cycle{1 2 3}.
	\end{displaymath}
	In \cite[Theorem 2.22]{Commuting_graph_T_X} it was mentioned that the authors verified through computations that $\dist{\commgraph{\tr{X}}}{\alpha}{\beta}\geqslant 5$. The aim of this proof is to see that we also have $\dist{\commgraph{\ptr{X}}}{\alpha}{\beta}\geqslant 5$.
	
	Let
	\begin{displaymath}
		\alpha=\gamma_1-\gamma_2-\cdots-\gamma_n=\beta
	\end{displaymath}
	be a path of minimum length from $\alpha$ to $\beta$ (in $\commgraph{\ptr{X}}$). It is clear that $\dist{\commgraph{\ptr{X}}}{\alpha}{\beta}=n-1$. In what follows we will progressively establish that $\dist{\commgraph{\ptr{X}}}{\alpha}{\beta}$ cannot be $1$, $2$, $3$ or $4$.
	
	We have that $\dist{\commgraph{\ptr{X}}}{\alpha}{\beta}\geqslant 2$ (and $n\geqslant 3$) because $2\alpha\beta=3\beta=1\neq 4=3\alpha=2\beta\alpha$.
	
	Since $\alpha$ is a cycle of length $8=\abs{X}$, then, by Lemma~\ref{P(X): what commutes with cycles length |X|}, the elements of $\ptr{X}$ that commute with $\alpha$ are precisely the ones in the set $\set{\emptyset,\id{X},\alpha,\alpha^2,\alpha^3,\alpha^4,\alpha^5,\alpha^6,\alpha^7}$. Hence $\gamma_2\in\set{\alpha^2,\alpha^3,\alpha^4,\alpha^5,\alpha^6,\alpha^7}\subseteq\sym{X}$ (because $\gamma_2$ commutes with $\alpha$ and $\gamma_2\in\ptr{X}\setminus\set{\emptyset,\id{X},\alpha}$). Additionally, we have that $\gamma_{n-2}\in\ptr{X}\setminus\set{\id{X}}$ and it commutes with $\beta$, which implies, by Lemma~\ref{P(X): upper bound, chain + cycle must commute with id-->S(X)}, that $\gamma_{n-1}\in\ptr{X}\setminus\sym{X}$. Then we must have $\gamma_2\neq\gamma_{n-1}$ and, consequently, we have that $n\geqslant 4$ and $\dist{\commgraph{\ptr{X}}}{\alpha}{\beta}\geqslant 3$.
	
	Let $s\in\mathbb{N}$ be such that $e=\gamma_{n-1}^s$ is an idempotent. Since $\gamma_{n-1}\in\tr{X}\setminus\sym{X}$, then we have $e=\gamma_{n-1}^s\neq\id{X}$. Furthermore, since $\gamma_{n-1}$ commutes with $\beta$, we have that $e=\gamma_{n-1}^s$ commutes with $\beta$, which means we can use Lemma~\ref{P(X): idempotent that commutes with chain+cycle} to conclude that
	\begin{displaymath}
		e=\chain{\set{1,8}, 1}\chain{\set{2,5,7}, 2}\chain{\set{3,4,6}, 3}.
	\end{displaymath}
	
	Since $\alpha=\gamma_1-\gamma_2-\cdots-\gamma_n=\beta$ is a path from $\alpha$ to $\beta$ of minimum length, we have that $e$ is distinct from $\gamma_1,\ldots,\gamma_{n-2}$ (because, otherwise, there would exist a smaller path from $\alpha$ to $\beta$). Therefore
	\begin{displaymath}
		\alpha=\gamma_1-\gamma_2-\cdots-\gamma_{n-2}-e-\gamma_n=\beta
	\end{displaymath}
	is also path of minimum length from $\alpha$ to $\beta$ (in $\commgraph{\ptr{X}}$). (Note that $e=\gamma_{n-1}^s$ commutes with $\gamma_{n-2}$ and $\beta$ because $\gamma_{n-1}$ commutes with $\gamma_{n-1}$ and $\beta$.)
	
	It follows from the fact that $\dist{\commgraph{\tr{X}}}{\alpha}{\beta}\geqslant 5$ that there are no paths in $\commgraph{\tr{X}}$ from $\alpha$ to $\beta$ of length $3$. Consequently, there are no paths in $\commgraph{\ptr{X}}$ from $\alpha$ to $\beta$ whose vertices are all full transformations and whose length is $3$, which implies that $\alpha-\gamma_2-e-\beta$ is not a path in $\commgraph{\ptr{X}}$ (because $\alpha,\gamma_2,e,\beta\in\tr{X}\setminus\set{\id{X}}$). Hence $n\geqslant 5$ and $\dist{\commgraph{\ptr{X}}}{\alpha}{\beta}\geqslant 4$.
	
	Finally, we are going to prove that $\dist{\commgraph{\ptr{X}}}{\alpha}{\beta}\geqslant 5$. It is enough to demonstrate that there is no vertex that is simultaneously adjacent to $\gamma_2$ and $e$. Let $\gamma\in\ptr{X}$ be such that $\gamma_2\gamma=\gamma\gamma_2$ and $e\gamma=\gamma e$. We are going to see that $\gamma\in\set{\emptyset,\id{X}}$. We have $\gamma\in\ptr{X}\setminus\tr{X}$ or $\gamma\in\tr{X}$, which motivates the division of the rest of the proof into two cases.
	
	\smallskip
	
	\textit{Case 1:} Suppose that $\gamma\in\ptr{X}\setminus\tr{X}$. We are going to see that $\graphptr{\gamma_2}{e}$ is connected. Since $\gamma_2\in\set{\alpha^2,\alpha^3,\alpha^4,\alpha^5,\alpha^6,\alpha^7}$, then, if we prove that $\graphptr{\alpha^i}{e}$ is connected for all $i\in\set{2,3,4,5,6,7}$, we can conclude that $\graphptr{\gamma_2}{e}$ is also connected. Figure~\ref{P(X), Figure: unified graph alpha2 alpha6, e (n=8)} illustrates the graph $\graphptr{\alpha^2}{e}=\graphptr{\alpha^6}{e}$, Figure~\ref{P(X), Figure: unified graph alpha3 alpha5, e (n=8)} illustrates the graph $\graphptr{\alpha^3}{e}=\graphptr{\alpha^5}{e}$, Figure~\ref{P(X), Figure: unified graph alpha4, e (n=8)} illustrates the graph $\graphptr{\alpha^4}{e}$, and Figure~\ref{P(X), Figure: unified graph alpha7, e (n=8)} illustrates the graph $\graphptr{\alpha^7}{e}$. If we analyse these figures, we can easily conclude that $\graphptr{\alpha^i}{e}$ is connected for all $i\in\set{2,3,4,5,6,7}$. Thus $\graphptr{\gamma_2}{e}$ is connected and, consequently, Lemma~\ref{P(X): unified graph lemma} implies that $\gamma=\emptyset$.
	\begin{figure}[hbt]
		\begin{center}
			\begin{tikzpicture}
				\foreach \x in {0,-1,-2,-3,-4,-5,-6,-7} {
					\node[vertex] (x\x) at (90+\x*360/8:1cm) {};
				}   
				
				\draw (90-0*360/8:1.4cm) node {$1$};
				\draw (90-1*360/8:1.4cm) node {$3$};
				\draw (90-2*360/8:1.4cm) node {$4$};
				\draw (90-3*360/8:1.4cm) node {$6$};
				\draw (90-4*360/8:1.4cm) node {$8$};
				\draw (90-5*360/8:1.4cm) node {$2$};
				\draw (90-6*360/8:1.4cm) node {$5$};
				\draw (90-7*360/8:1.4cm) node {$7$};
				
				\begin{scope}[edge]
					\draw (0,0) circle (1cm);
					\draw (x0) -- (x-4);
					\draw (x-1) -- (x-3);
					\draw (x-1) -- (x-6);
					\draw (x-5) -- (x-2);
					\draw (x-5) -- (x-7);
				\end{scope}
			\end{tikzpicture}
		\end{center}
		\caption{Graph $\graphptr{\alpha^2}{e}=\graphptr{\alpha^6}{e}$ where $\alpha=\cycle{1 2 3 4 5 6 7 8}$ and $e=\chain{\set{1,8}, 1}\allowbreak\chain{\set{2,5,7}, 2}\allowbreak\chain{\set{3,4,6}, 3}$.}
		\label{P(X), Figure: unified graph alpha2 alpha6, e (n=8)}
	\end{figure}
	\begin{figure}[hbt]
		\begin{center}
			\begin{tikzpicture}
				\foreach \x in {0,-1,-2,-3,-4,-5,-6,-7} {
					\node[vertex] (x\x) at (90+\x*360/8:1cm) {};
				}   
				
				\draw (90-0*360/8:1.4cm) node {$1$};
				\draw (90-1*360/8:1.4cm) node {$4$};
				\draw (90-2*360/8:1.4cm) node {$7$};
				\draw (90-3*360/8:1.4cm) node {$2$};
				\draw (90-4*360/8:1.4cm) node {$5$};
				\draw (90-5*360/8:1.4cm) node {$8$};
				\draw (90-6*360/8:1.4cm) node {$3$};
				\draw (90-7*360/8:1.4cm) node {$6$};
				
				\begin{scope}[edge]
					\draw (0,0) circle (1cm);
					\draw (x0) -- (x-5);
					\draw (x-1) -- (x-6);
				\end{scope}
			\end{tikzpicture}
		\end{center}
		\caption{Graph $\graphptr{\alpha^3}{e}=\graphptr{\alpha^5}{e}$ where $\alpha=\cycle{1 2 3 4 5 6 7 8}$ and $e=\chain{\set{1,8}, 1}\allowbreak\chain{\set{2,5,7}, 2}\allowbreak\chain{\set{3,4,6}, 3}$.}
		\label{P(X), Figure: unified graph alpha3 alpha5, e (n=8)}
	\end{figure}
	\begin{figure}[hbt]
		\begin{center}
			\begin{tikzpicture}
				\foreach \x in {0,-1,-2,-3,-4,-5,-6} {
					\node[vertex] (x\x) at (90+\x*360/7:1cm) {};
				}   
				
				\draw (90-0*360/7:1.4cm) node {$1$};
				\draw (90-1*360/7:1.4cm) node {$5$};
				\draw (90-2*360/7:1.4cm) node {$2$};
				\draw (90-3*360/7:1.4cm) node {$7$};
				\draw (90-4*360/7:1.4cm) node {$3$};
				\draw (90-5*360/7:1.4cm) node {$4$};
				\draw (90-6*360/7:1.4cm) node {$8$};

				\node[vertex] (6) at (0,0) {};
				\node[anchor=south,inner sep=2mm] at (6) {$6$};
				
				\begin{scope}[edge]
					\draw (0,0) circle (1cm);
					\draw (x-2) -- (6);
					\draw (x-4) -- (6);
				\end{scope}
			\end{tikzpicture}
		\end{center}
		\caption{Graph $\graphptr{\alpha^4}{e}$ where $\alpha=\cycle{1 2 3 4 5 6 7 8}$ and $e=\chain{\set{1,8}, 1}\allowbreak\chain{\set{2,5,7}, 2}\allowbreak\chain{\set{3,4,6}, 3}$.}
		\label{P(X), Figure: unified graph alpha4, e (n=8)}
	\end{figure}
	\begin{figure}[h!bt]
		\begin{center}
			\begin{tikzpicture}
				\foreach \x in {0,-1,-2,-3,-4,-5,-6,-7} {
					\node[vertex] (x\x) at (90+\x*360/8:1cm) {};
				}   
				
				\draw (90-0*360/8:1.4cm) node {$1$};
				\draw (90-1*360/8:1.4cm) node {$8$};
				\draw (90-2*360/8:1.4cm) node {$7$};
				\draw (90-3*360/8:1.4cm) node {$6$};
				\draw (90-4*360/8:1.4cm) node {$5$};
				\draw (90-5*360/8:1.4cm) node {$4$};
				\draw (90-6*360/8:1.4cm) node {$3$};
				\draw (90-7*360/8:1.4cm) node {$2$};
				
				\begin{scope}[edge]
					\draw (0,0) circle (1cm);
					\draw (x-7) -- (x-2);
					\draw (x-7) -- (x-4);
					\draw (x-3) -- (x-6);
				\end{scope}
			\end{tikzpicture}
		\end{center}
		\caption{Graph $\graphptr{\alpha^7}{e}$ where $\alpha=\cycle{1 2 3 4 5 6 7 8}$ and $e=\chain{\set{1,8}, 1}\allowbreak\chain{\set{2,5,7}, 2}\allowbreak\chain{\set{3,4,6}, 3}$.}
		\label{P(X), Figure: unified graph alpha7, e (n=8)}
	\end{figure}
	
	\smallskip
	
	\textit{Case 2:} Suppose that $\gamma\in\tr{X}$. We know that, in $\commgraph{\tr{X}}$, there are no paths from $\alpha$ to $\beta$ of length $4$ (since $\dist{\commgraph{\tr{X}}}{\alpha}{\beta}\geqslant 5$). Hence, in $\commgraph{\ptr{X}}$, there are no paths from $\alpha$ to $\beta$ whose vertices are all full transformations and whose length is $4$. Then $\alpha-\gamma_2-\gamma-e-\beta$ is not a path and, thus, we must have $\gamma=\id{X}$.
	
	\smallskip
	
	In both cases we concluded that $\gamma\in\set{\emptyset,\id{X}}$, which allows us to conclude that $\commgraph{\ptr{X}}$ contains no vertex adjacent to both $\gamma_2$ and $e$. Therefore $n\geqslant 6$ and $\dist{\commgraph{\ptr{X}}}{\alpha}{\beta}\geqslant 5$, which concludes the proof.
\end{proof}

\begin{lemma}\label{P(X): lower bound diam, |X|>=9 composite}
	Suppose that $\abs{X}\geqslant 9$ is composite. Then there exist $\alpha,\beta\in\ptr{X}\setminus\set{\emptyset,\id{X}}$ such that $\dist{\commgraph{\ptr{X}}}{\alpha}{\beta}\geqslant 5$.
\end{lemma}

\begin{proof}
	We divide the proof into two parts: the first part concerns the case where $\abs{X}$ is odd and the second part concerns the case where $\abs{X}$ is even.
	
	\medskip
	
	\textbf{Part 1.} Suppose that $\abs{X}=2m+1$ for some $m\geqslant 4$. Assume that $X=\set{x_1,\ldots,x_m}\cup\set{y_1,\ldots y_m}\cup\set{z}$ and let
	\begin{displaymath}
		\alpha=\chain{z y_1 y_2 {\ldots} y_m x_1}\cycle{x_1 x_2 {\ldots} x_m}
		\quad \text{and} \quad
		\beta=\chain{x_2 x_3 {\ldots} x_m x_1 z y_2}\cycle{y_2 {\ldots} y_m y_1}.
	\end{displaymath}
	
	In the proof of \cite[Theorem 2.22]{Commuting_graph_T_X} the authors showed that $\dist{\commgraph{\tr{X}}}{\alpha}{\beta}\geqslant 5$. Our goal is to prove that we also have $\dist{\commgraph{\ptr{X}}}{\alpha}{\beta}\geqslant 5$.
	
	Let
	\begin{displaymath}
		\alpha=\gamma_1-\gamma_2-\cdots-\gamma_n=\beta
	\end{displaymath}
	a path of minimum length from $\alpha$ to $\beta$ (in $\commgraph{\ptr{X}}$). Then $\dist{\commgraph{\ptr{X}}}{\alpha}{\beta}=n-1$.
	
	We observe that, since $\dist{\commgraph{\tr{X}}}{\alpha}{\beta}\geqslant 5$, then $\alpha$ and $\beta$ are not adjacent in $\commgraph{\tr{X}}$; that is, $\alpha\beta\neq\beta\alpha$. Thus $\alpha$ and $\beta$ are not adjacent in $\commgraph{\ptr{X}}$ and, thus $\dist{\commgraph{\ptr{X}}}{\alpha}{\beta}\geqslant 2$ (and $n\geqslant 3$).
	
	
	The fact that $\gamma_2$ and $\gamma_{n-1}$ are vertices of $\commgraph{\ptr{X}}$ implies that $\gamma_2,\gamma_{n-1}\in\ptr{X}\setminus\set{\emptyset,\id{X}}$. Moreover, we have that $\alpha\gamma_2=\gamma_2\alpha$ and $\beta\gamma_{n-1}=\gamma_{n-1}\beta$. By Lemma~\ref{P(X): upper bound, chain + cycle must commute with id-->S(X)}, we have that $\gamma_2,\gamma_{n-1}\in \ptr{X}\setminus\sym{X}$ and, by Lemma~\ref{P(X): upper bound, chain + cycle must commute with empty map-->P(X)-T(X)}, we have that $\gamma_2,\gamma_{n-1}\in\tr{X}$. Thus $\gamma_2,\gamma_{n-1}\in\tr{X}\setminus\sym{X}$. Additionally, there exist $s,t\in\mathbb{N}$ such that $e=\gamma_2^s$ and $f=\gamma_{n-1}^t$ are idempotents. Due to the fact $\gamma_2,\gamma_{n-1}\in\tr{X}\setminus\sym{X}$, we can conclude that $e,f\in\tr{X}\setminus\set{\id{X}}$. Moreover, we have $\alpha e=e\alpha$ and $\beta f=f\beta$ (because $\alpha\gamma_2=\gamma_2\alpha$ and $\beta\gamma_{n-1}=\gamma_{n-1}\beta$). Hence, by Lemma~\ref{P(X): idempotent that commutes with chain+cycle}, we must have
	\begin{gather*}
		e=\chain{\set{x_1,y_1}, x_1}\cdots\chain{\set{x_{m-1},y_{m-1}}, x_{m-1}}\chain{\set{x_m,y_m,z}, x_m}\\
		\shortintertext{and}
		f=\chain{\set{y_1,x_2,z}, y_1}\chain{\set{y_2,x_3}, y_2}\cdots\chain{\set{y_{m-1},x_m}, y_{m-1}}\chain{\set{y_m,x_1}, y_m}.
	\end{gather*}
	
	It is clear that $e\neq f$. This can be proved simply by verifying that $\im e=\set{x_1,\ldots,x_m}\neq\set{y_1,\ldots,y_m}=\im f$. Hence $\gamma_2\neq\gamma_{n-1}$ and, consequently, we must have $n\geqslant 4$ and $\dist{\commgraph{\ptr{X}}}{\alpha}{\beta}\geqslant 3$.

	Since $\gamma_2$ commutes with $\alpha$ and $\gamma_3$, and $\gamma_{n-1}$ commutes with $\beta$ and $\gamma_{n-2}$, then we also have that $e$ commutes $\alpha$ and $\gamma_3$, and $f$ commutes with $\beta$ and $\gamma_{n-2}$. Moreover, since the path we are considering is of minimum length, it is clear that none of the partial transformations $\gamma_3,\ldots,\gamma_n$ is equal to $e$, and none of the partial transformations $\gamma_1,\ldots,\gamma_{n-2}$ is equal to $f$. Thus
	\begin{displaymath}
		\alpha=\gamma_1-e-\gamma_3-\cdots-\gamma_{n-2}-f-\gamma_n=\beta
	\end{displaymath}
	is also a path of minimum length from $\alpha$ to $\beta$ (in $\commgraph{\ptr{X}}$).
	
	We observe that we have $n\geqslant 5$; that is, $\dist{\commgraph{\ptr{X}}}{\alpha}{\beta}\geqslant 4$, because $ef\neq fe$ (we notice that we have, for example, $x_1 ef=x_1f=y_m\neq x_m=y_m e=x_1 fe$). This means that, in the path above, there must be at least one element between $e$ and $f$.
	
	Now we are going to see that there is no vertex that is simultaneously adjacent to $e$ and $f$. Let $\gamma\in\ptr{X}$ be such that $e\gamma=\gamma e$ and $f\gamma=\gamma f$. We have $\gamma\in\ptr{X}\setminus\tr{X}$ or $\gamma\in\tr{X}$.
	
	\smallskip
	
	\textit{Case 1:} Suppose that $\gamma\in\ptr{X}\setminus\tr{X}$. We have that $\graphptr{e}{f}$ is connected --- this can be verified by observation of Figure~\ref{P(X), Figure: unified graph e, f odd case}, which shows the graph $\graphptr{e}{f}$. Then, by Lemma~\ref{P(X): unified graph lemma}, we have that $\gamma=\emptyset$ and, consequently, $\gamma$ is not a vertex of $\commgraph{\ptr{X}}$.
	\begin{figure}[hbt]
		\begin{center}
			\begin{tikzpicture}
				\foreach \x in {0,-1,-2,-3,-4,-6,-7,-8} {
					\node[vertex] (x\x) at (90+\x*360/9:1cm) {};
				}
				
				\draw (90:1.4cm) node {$x_1$};
				\draw (90-1*360/9:1.4cm) node {$y_1$};
				\draw (90-2*360/9:1.4cm) node {$x_2$};
				\draw (90-3*360/9:1.4cm) node {$y_2$};
				\draw (90-4*360/9:1.4cm) node {$x_3$};
				\draw[xshift=-1mm] (90-6*360/9:1.4cm) node {$y_{m-1}$};
				\draw (90-7*360/9:1.4cm) node {$x_m$};
				\draw (90-8*360/9:1.4cm) node {$y_m$};
				
				\node[vertex] (z) at (0,0) {};
				\node[anchor=north,inner sep=2mm] at (z) {$z$};
				
				\begin{scope}[edge]
					\draw (90:1cm) arc (90:90-1*360/9:1cm);
					\draw (90-1*360/9:1cm) arc (90-1*360/9:90-2*360/9:1cm);
					\draw (90-2*360/9:1cm) arc (90-2*360/9:90-3*360/9:1cm);
					\draw (90-3*360/9:1cm) arc (90-3*360/9:90-4*360/9:1cm);
					\draw[dotted] (90-4*360/9:1cm) arc (90-4*360/9:90-6*360/9:1cm);
					\draw (90-6*360/9:1cm) arc (90-6*360/9:90-7*360/9:1cm);
					\draw (90-7*360/9:1cm) arc (90-7*360/9:90-8*360/9:1cm);
					\draw (90-8*360/9:1cm) arc (90-8*360/9:90-9*360/9:1cm);
					
					\draw (z) -- (x-1);
					\draw (z) -- (x-7);
				\end{scope}
			\end{tikzpicture}
		\end{center}
		\caption{Graph $\graphptr{e}{f}$ where $e=\chain{\set{x_1,y_1}, x_1}\cdots\allowbreak\chain{\set{x_{m-1},y_{m-1}}, x_{m-1}}\allowbreak\chain{\set{x_m,y_m,z}, x_m}$ and $f=\chain{\set{y_1,x_2,z}, y_1}\allowbreak\chain{\set{y_2,x_3}, y_2}\cdots\allowbreak\chain{\set{y_{m-1},x_m}, y_{m-1}}\allowbreak\chain{\set{y_m,x_1}, y_m}$.}
		\label{P(X), Figure: unified graph e, f odd case}
	\end{figure}
	
	\smallskip
	
	\textit{Case 2:} Suppose that $\gamma\in\tr{X}$. We have mentioned before that $\dist{\commgraph{\tr{X}}}{\alpha}{\beta}\geqslant 5$. This implies that, in $\commgraph{\tr{X}}$, there are no paths from $\alpha$ to $\beta$ of length at most $4$. Since $\commgraph{\tr{X}}$ is a subgraph of $\commgraph{\ptr{X}}$, then, in $\commgraph{\ptr{X}}$, there are no paths from $\alpha$ to $\beta$ whose vertices are all full transformations and whose length is at most $4$. Hence we must have $\gamma=\id{X}$ (because, otherwise, $\alpha-e-\gamma-f-\beta$ would be a path of length $4$ formed by full transformations). Consequently, $\gamma$ is not a vertex of $\commgraph{\ptr{X}}$.
	
	In both cases we established that $\gamma$ is not a vertex of $\commgraph{\ptr{X}}$. Hence there is no vertex that is adjacent to $e$ and $f$. This implies that $\gamma_3\neq\gamma_{n-2}$; that is, $n\geqslant 6$. Thus $\dist{\commgraph{\ptr{X}}}{\alpha}{\beta}\geqslant 5$, which concludes the proof of part 1.
	
	
	
	\medskip
	
	\textbf{Part 2.} Suppose that $\abs{X}=2m+2$ for some $m\geqslant 4$. Assume that $X=\set{x_1,\ldots,x_m}\cup\set{y_1,\ldots y_m}\cup\set{z,w}$. In the proof of \cite[Theorem 2.22]{Commuting_graph_T_X} the authors showed that $\dist{\commgraph{\tr{X}}}{\alpha}{\beta}\geqslant 5$, where
	\begin{gather*}
		\alpha=\chain{z y_1 y_2 {\ldots} y_m w x_2}\cycle{x_2 {\ldots} x_m x_1}\\
		\shortintertext{and}
		\beta=\chain{w x_2 x_3 {\ldots} x_{m-2} x_m x_1 x_{m-1} z y_2}\cycle{y_2 {\ldots} y_m y_1}.\mbox{\footnotemark[1]}
	\end{gather*}
	\footnotetext[1]{Although in \cite[Theorem 2.22]{Commuting_graph_T_X} the authors defined $\beta$ as the full transformation $\chain{w x_2 x_3 {\ldots} x_{m-2} x_m x_1 x_{m-1} y_2}\allowbreak\cycle{y_2 {\ldots} y_m y_1}$ (where the image of $z$ is not defined), if we follow the proof closely we can see that the full transformation they meant is actually $\chain{w x_2 x_3 {\ldots} x_{m-2} x_m x_1 x_{m-1} z y_2}\allowbreak\cycle{y_2 {\ldots} y_m y_1}$ --- the one we are using.}
	Our objective is to establish that $\dist{\commgraph{\ptr{X}}}{\alpha}{\beta}\geqslant 5$.
	
	Let
	\begin{displaymath}
		\alpha=\gamma_1-\gamma_2-\cdots-\gamma_n=\beta
	\end{displaymath}
	be a path of minimum length from $\alpha$ to $\beta$ (in $\commgraph{\ptr{X}}$). It is clear that $\dist{\commgraph{\ptr{X}}}{\alpha}{\beta}=n-1$.
	
	Due to the fact that $\alpha$ and $\beta$ are not adjacent in $\commgraph{\ptr{X}}$, then they are also not adjacent in $\commgraph{\ptr{X}}$ and, consequently, we have $\dist{\commgraph{\ptr{X}}}{\alpha}{\beta}\geqslant 2$ (and $n\geqslant 3$).
	
	
	We know that there exist $s,t\in\mathbb{N}$ such that $e=\gamma_2^s$ and $f=\gamma_{n-1}^t$ are idempotents. Moreover, we have that $\gamma_2,\gamma_{n-1}\in\ptr{X}\setminus\set{\emptyset,\id{X}}$ and $\alpha\gamma_2=\gamma_2\alpha$ and $\beta\gamma_{n-1}=\gamma_{n-1}\beta$. Hence Lemma~\ref{P(X): upper bound, chain + cycle must commute with id-->S(X)} implies that $\gamma_2,\gamma_{n-1}\in \ptr{X}\setminus\sym{X}$, and Lemma~\ref{P(X): upper bound, chain + cycle must commute with empty map-->P(X)-T(X)} implies that $\gamma_2,\gamma_{n-1}\in\tr{X}$. Therefore $\gamma_2,\gamma_{n-1}\in\tr{X}\setminus\sym{X}$ and, consequently, we must have $e,f\in\tr{X}\setminus\set{\id{X}}$. In addition, we have $\alpha e=e\alpha$ and $\beta f=f\beta$ (because $\alpha\gamma_2=\gamma_2\alpha$ and $\beta\gamma_{n-1}=\gamma_{n-1}\beta$). Hence, by Lemma~\ref{P(X): idempotent that commutes with chain+cycle}, we must have
	\begin{displaymath}
		e=\chain{\set{x_1,y_1,w}, x_1}\chain{\set{x_2,y_2}, x_2}\cdots\chain{\set{x_{m-1},y_{m-1}}, x_{m-1}}\chain{\set{x_m,y_m,z}, x_m}
	\end{displaymath}
	and
	\begin{multline*}
		f=\chain{\set{y_1,x_2,z}, y_1}\chain{\set{y_2,x_3}, y_2}\cdots\chain{\set{y_{m-3},x_{m-2}}, y_{m-3}}\\
		\chain{\set{y_{m-2},x_m}, y_{m-2}}\chain{\set{y_{m-1},x_1}, y_{m-1}}\chain{\set{y_m,x_{m-1},w}, y_m}.
	\end{multline*}
	
	It follows from the fact that $\im e=\set{x_1,\ldots,x_m}\neq\set{y_1,\ldots,y_m}=\im f$ that $\gamma_2^s=e\neq f=\gamma_{n-1}^t$. This implies that $\gamma_2\neq\gamma_{n-1}$ and, consequently, we must have $n\geqslant 4$ and $\dist{\commgraph{\ptr{X}}}{\alpha}{\beta}\geqslant 3$.

	Since $\gamma_2$ commutes with $\alpha$ and $\gamma_3$, and $\gamma_{n-1}$ commutes with $\beta$ and $\gamma_{n-2}$, then we also have that $e$ commutes $\alpha$ and $\gamma_3$, and $f$ commutes with $\beta$ and $\gamma_{n-2}$. Moreover, since the path we are considering is of minimum length, it is clear that none of the partial transformations $\gamma_3,\ldots,\gamma_n$ is equal to $e$, and none of the partial transformations $\gamma_1,\ldots,\gamma_{n-2}$ is equal to $f$. Thus
	\begin{displaymath}
		\alpha=\gamma_1-e-\gamma_3-\cdots-\gamma_{n-2}-f-\gamma_n=\beta
	\end{displaymath}
	is also a path of minimum length from $\alpha$ to $\beta$ (in $\commgraph{\ptr{X}}$).
	
	We note that $e$ and $f$ do not commute (because $x_1 ef=x_1f=y_{m-1}\neq x_{m-1}=y_{m-1} e=x_1 fe$). Hence $n\geqslant 5$ and $\dist{\commgraph{\ptr{X}}}{\alpha}{\beta}\geqslant 4$, which implies that, in the path above there is at least one element between $e$ and $f$.
	
	Let $\gamma\in\ptr{X}$ be such that $e\gamma=\gamma e$ and $f\gamma=\gamma f$. We want to see that $\gamma\in\set{\emptyset,\id{X}}$. We have two possibilities: $\gamma\in\ptr{X}\setminus\tr{X}$ or $\gamma\in\tr{X}$.
	
	\smallskip
	
	\textit{Case 1:} Suppose that $\gamma\in\ptr{X}\setminus\tr{X}$. In Figure~\ref{P(X), Figure: unified e, f even case} we have an illustration of the graph $\graphptr{e}{f}$, which is clearly connected. Consequently, Lemma~\ref{P(X): unified graph lemma} allows us to conclude that $\gamma=\emptyset$.
	\begin{figure}[hbt]
		\begin{center}
			\begin{tikzpicture}
				\foreach \x in {0,-1,-2,-3,-5,-6,-7,-8,-9,-10,-11,-12} {
					\node[vertex] (x\x) at (90+\x*360/13:1.5cm) {};
				}
				
				\draw (90:1.9cm) node {$x_1$};
				\draw (90-1*360/13:1.9cm) node {$y_1$};
				\draw (90-2*360/13:1.9cm) node {$x_2$};
				\draw (90-3*360/13:1.9cm) node {$y_2$};
				\draw[xshift=1mm] (90-5*360/13:1.9cm) node {$x_{m-3}$};
				\draw (90-6*360/13:1.9cm) node {$y_{m-3}$};
				\draw (90-7*360/13:1.9cm) node {$x_{m-2}$};
				\draw (90-8*360/13:1.9cm) node {$y_{m-2}$};
				\draw (90-9*360/13:1.9cm) node {$x_m$};
				\draw (90-10*360/13:1.9cm) node {$y_m$};
				\draw[xshift=-1mm] (90-11*360/13:1.9cm) node {$x_{m-1}$};
				\draw (90-12*360/13:1.9cm) node {$y_{m-1}$};
				
				\node[vertex] (z) at (90-3*360/8:0.6cm) {};
				\node[vertex] (w) at (90-7*360/8:0.5cm) {};
				
				\node[anchor=north west] at (z) {$z$};
				\node[anchor=north west] at (w) {$w$};
				
				\begin{scope}[edge]
					\draw (90+8*360/13:1.5cm) arc (90+8*360/13:90-3*360/13:1.5cm);
					\draw[dotted] (90-3*360/13:1.5cm) arc (90-3*360/13:90-5*360/13:1.5cm);
					
					\draw (z) -- (x-1);
					\draw (z) -- (x-9);
					\draw (w) -- (x0);
					\draw (w) -- (x-10);
				\end{scope}
			\end{tikzpicture}
		\end{center}
		\caption{Graph $\graphptr{e}{f}$ where $e=\chain{\set{x_1,y_1,w}, x_1}\allowbreak\chain{\set{x_2,y_2}, x_2}\cdots\allowbreak\chain{\set{x_{m-1},y_{m-1}}, x_{m-1}}\allowbreak\chain{\set{x_m,y_m,z}, x_m}$ and $f=\chain{\set{y_1,x_2,z}, y_1}\allowbreak\chain{\set{y_2,x_3}, y_2}\cdots\allowbreak\chain{\set{y_{m-3},x_{m-2}}, y_{m-3}}\allowbreak\chain{\set{y_{m-2},x_m}, y_{m-2}}\allowbreak\chain{\set{y_{m-1},x_1}, y_{m-1}}\allowbreak\chain{\set{y_m,x_{m-1},w}, y_m}$.}
		\label{P(X), Figure: unified e, f even case}
	\end{figure}
	
	\smallskip

	\textit{Case 2:} Suppose that $\gamma\in\tr{X}$. Since $\dist{\commgraph{\tr{X}}}{\alpha}{\beta}\geqslant 5$, then there is no path of length $4$ from $\alpha$ to $\beta$ in $\commgraph{\tr{X}}$. Consequently, there is no path of length $4$ from $\alpha$ to $\beta$ in $\commgraph{\ptr{X}}$ whose vertices are all full transformations. Therefore $\alpha-e-\gamma-f-\beta$ cannot be a path in $\commgraph{\ptr{X}}$, which implies that $\gamma\in\centre{\ptr{X}}\cap\tr{X}=\set{\id{X}}$.
	
	\smallskip
	
	It follows from cases 1 and 2 that $\gamma\in\set{\emptyset,\id{X}}=\centre{\ptr{X}}$, which implies that there is no vertex that is simultaneously adjacent to $e$ and $f$. Then $\gamma_3\neq\gamma_{n-2}$ and, consequently, $n\geqslant 6$; that is, $\dist{\commgraph{\ptr{X}}}{\alpha}{\beta}\geqslant 5$, which concludes part 2 of the proof.
\end{proof}

At last we can use the lemmata proved on the previous pages to obtain the diameter of $\commgraph{\ptr{X}}$.

\begin{theorem}\label{P(X): diameter}
	Suppose that $\abs{X}\geqslant 2$. Then
	\begin{enumerate}
		\item If $\abs{X}$ is prime, then $\commgraph{\ptr{X}}$ is not connected.
		
		\item If $\abs{X}=4$, then $\diam{\commgraph{\ptr{X}}}=4$.
		
		\item If $\abs{X}\geqslant 6$ is composite, then $\diam{\commgraph{\ptr{X}}}=5$.
	\end{enumerate}
\end{theorem}

\begin{proof}
	\textbf{Part 1.} Suppose that $\abs{X}$ is prime.
	
	Let $\alpha\in\sym{X}$ be a cycle of length $\abs{X}$. Let $\beta\in \ptr{X}\setminus\set{\emptyset,\id{X},\alpha}$ be a vertex of $\commgraph{\ptr{X}}$ adjacent to $\alpha$. Then $\alpha\beta=\beta\alpha$ and, by Lemma~\ref{P(X): what commutes with cycles length |X|}, we have that $\beta\in\set{\alpha^2,\ldots,\alpha^{\abs{X}-1}}$. Since $\abs{X}$ is prime, we can conclude that the permutations $\alpha^2,\ldots,\alpha^{\abs{X}-1}$ are all cycles of length $\abs{X}$, which implies that $\beta$ is a cycle of length $\abs{X}$.
	
	Therefore the vertices of $\commgraph{\ptr{X}}$ that are cycles of length $\abs{X}$ can only be adjacent to vertices that are also cycles of length $\abs{X}$. Thus $\commgraph{\ptr{X}}$ has at least two connected components --- one of them only has vertices that are cycles of length $\abs{X}$, and at least one other only has vertices that are not cycles of length $\abs{X}$. Consequently, $\commgraph{\ptr{X}}$ is not connected.
	
	\medskip
	
	\textbf{Part 2.} Suppose that $\abs{X}\geqslant 4$ is composite.
	
	In what follows we are going to establish an upper bound for $\diam{\commgraph{\ptr{X}}}$. Let $\alpha,\beta\in\ptr{X}\setminus\set{\emptyset, \id{X}}$ be two vertices of $\commgraph{\ptr{X}}$.
	
	\smallskip
	
	\textit{Case 1:} Assume that $\alpha,\beta\in\tr{X}\setminus\set{\id{X}}$. Since $\commgraph{\tr{X}}$ is a subgraph of $\commgraph{\ptr{X}}$, then we have that
	\begin{displaymath}
		\dist{\commgraph{\ptr{X}}}{\alpha}{\beta}\leqslant\dist{\commgraph{\tr{X}}}{\alpha}{\beta}\leqslant\diam{\commgraph{\tr{X}}}=\begin{cases}
			4 &\text{if } \abs{X}=4,\\
			5 &\text{if } \abs{X}\neq 4
		\end{cases}
	\end{displaymath}
	where the equality at the end is given by Theorem~\ref{P(X): diam T(X)}.

	
	\smallskip
	
	\textit{Case 2:} Assume that $\alpha,\beta\in\ptr{X}\setminus\parens{\tr{X}\cup\set{\emptyset}}$. Then, by Lemma~\ref{P(X): upper bound diam, a,b€P(X)-(T(X) U 0)}, we have that $\dist{\commgraph{\ptr{X}}}{\alpha}{\beta}\leqslant 4$.
	
	\smallskip
	
	\textit{Case 3:} Assume that $\alpha\in\tr{X}\setminus\sym{X}$ and $\beta\in\ptr{X}\setminus\parens{\tr{X}\cup\set{\emptyset}}$ (or $\alpha\in\ptr{X}\setminus\parens{\tr{X}\cup\set{\emptyset}}$ and $\beta\in\tr{X}\setminus\sym{X}$). It follows from Lemma~\ref{P(X): upper bound diam, a€T(X)-S(X) b€P(X)-(T(X) U 0)} that $\dist{\commgraph{\ptr{X}}}{\alpha}{\beta}\leqslant 4$.
	
	\smallskip
	
	\textit{Case 4:} Assume that $\alpha\in\sym{X}\setminus\set{\id{X}}$ and $\beta\in\ptr{X}\setminus\parens{\tr{X}\cup\set{\emptyset}}$ (or $\alpha\in\ptr{X}\setminus\parens{\tr{X}\cup\set{\emptyset}}$ and $\beta\in\sym{X}\setminus\set{\id{X}}$). Then Lemma~\ref{P(X): upper bound diam, a€S(X)-1 b€P(X)-(T(X) U 0)} establishes that
	\begin{displaymath}
		\dist{\commgraph{\ptr{X}}}{\alpha}{\beta}\leqslant\begin{cases}
			4 &\text{if } \abs{X}=4\\
			5 &\text{if } \abs{X}\neq 4.
		\end{cases} 
	\end{displaymath}
	
	\smallskip
	
	Since $\alpha$ and $\beta$ are arbitrary vertices of $\commgraph{\ptr{X}}$, then cases 1--4 allow us to conclude that
	\begin{displaymath}
		\diam{\commgraph{\ptr{X}}}=\max\gset{\dist{\commgraph{\ptr{X}}}{\alpha}{\beta}}{\alpha,\beta\in\ptr{X}\setminus\set{\emptyset,\id{X}}}\leqslant\begin{cases}
			4 &\text{if } \abs{X}=4,\\
			5 &\text{if } \abs{X}\neq 4.
		\end{cases} 
	\end{displaymath}
	
	Moreover, Lemma~\ref{lower bound diam, |X|=4} establishes that, when $\abs{X}=4$, then there exist $\alpha,\beta\in\ptr{X}\setminus\set{\emptyset,\id{X}}$ such that $\dist{\commgraph{\ptr{X}}}{\alpha}{\beta}\geqslant 4$; and Lemmata~\ref{P(X): lower bound diam, |X|=6 composite}, \ref{P(X): lower bound diam, |X|=8 composite}, \ref{P(X): lower bound diam, |X|>=9 composite} establish that, when $\abs{X}\neq 4$, then there exist $\alpha,\beta\in\ptr{X}\setminus\set{\emptyset,\id{X}}$ such that $\dist{\commgraph{\ptr{X}}}{\alpha}{\beta}\geqslant 5$. Therefore
	\begin{displaymath}
		\diam{\commgraph{\ptr{X}}}=\begin{cases}
			4 &\text{if } \abs{X}=4,\\
			5 &\text{if } \abs{X}\neq 4
		\end{cases}
	\end{displaymath}
	which concludes the proof.
\end{proof}

By comparison of Theorems~\ref{P(X): diam T(X)} and \ref{P(X): diameter}, we can conclude that $\diam{\commgraph{\ptr{X}}}=\diam{\commgraph{\tr{X}}}$. Since $\tr{X}$ is a subsemigroup of $\ptr{X}$, this shows that there are semigroups whose commuting graph has the same diameter as the commuting graph of one of its (proper) subsemigroups. Moreover, it is also possible to find a semigroup $S$ and a subsemigroup $T$ of $S$ such that $\commgraph{T}$ and $\commgraph{S}$ are connected and $\diam{\commgraph{T}}<\diam{\commgraph{S}}$. For instance, the symmetric inverse semigroup $\psym{X}$ on $X$ is a subsemigroup of $\ptr{X}$ and, when $\abs{X}\geqslant 6$ is even, we have $\diam{\commgraph{\psym{X}}}=4<5=\diam{\commgraph{\ptr{X}}}$ (the diameter of $\psym{X}$ is obtained in \cite[Theorem 6.12]{Commuting_graph_I_X}).

    \bibliography{Bibliography} 
\bibliographystyle{alphaurl}

\end{document}